\documentclass[11pt]{article}

\pdfoutput=1
\usepackage{mathtools,mathrsfs}
\usepackage{amsmath, amsfonts,amsthm,mathabx}
\usepackage{graphicx, xcolor}
\usepackage{multirow,enumerate}
\usepackage{tikz,pgfplots}
\pgfplotsset{compat=1.10}
\usepackage{booktabs,colortbl}
\usepackage{arydshln}
\newcommand{\ra}[1]{\renewcommand{\arraystretch}{#1}}
\usepackage{mathrsfs,xspace}
\usepackage{algorithmic}

\usepackage{blindtext}
\usepackage{rotating}
\usepackage[]{cite}
\usepackage{float,rotating}

  \usepackage{url}
\usepackage{hyperref,fullpage}
\usepackage[capitalize]{cleveref}

\Crefname{ALC@unique}{Line}{Lines}

  \newcommand{\funding}[1]{{\bf Funding:} #1}

  \newcommand{\email}[1]{\url{#1}}

\newtheorem{remark}{Remark}
\newtheorem{theorem}{Theorem}
\newtheorem{lemma}{Lemma}

\newcommand{\scinote}[3]{#1\text{e$#2$#3}}

\newcommand{\R}{\mathbf{R}}

\newcommand{\Z}{\mathbf{Z}}

\newcommand{\J}{\mathcal{J}}

\renewcommand{\O}{O}

\usepackage{mathtools,mathabx}
\usepackage{systeme}

\newcommand{\FL}{(-\Delta)^{\alpha/2}}

\renewcommand{\v}[1]{{\mathbf{#1}}}
\newcommand{\m}[1]{{\mathsf{#1}}}

\newcommand{\I}{\text{(I)}}
\newcommand{\II}{\text{(II)}}
\newcommand{\Id}{\text{(Id)}}
\newcommand{\IId}{\text{(IId)}}

\newcommand*\from{\colon}

  \newcommand{\comp}[1]{{#1}^\mathsf{c}}

\title{A simple solver for the fractional Laplacian in multiple dimensions}

\author{ Victor Minden\thanks{Institute for Computational and Mathematical Engineering, Stanford
    University, Stanford, CA 94305.  Current address: Center for Computational Biology, Flatiron
    Institute, Simons Foundation, New York, NY 10017
    (\email{vminden@flatironinstitute.org}). \funding{U.S. Department of Energy Advanced Scientific
      Computing Research program (grant number DE-FC02-13ER26134/DE-SC0009409).}}
  \and
  Lexing Ying\thanks{Department of Mathematics and Institute for Computational and Mathematical
    Engineering, Stanford University, Stanford, CA 94305
    (\email{lexing@stanford.edu}).\funding{National Science Foundation (grant number DMS-1521830)
      and U.S. Department of Energy Advanced Scientific Computing Research program (grant number
      DE-FC02-13ER26134/DE-SC0009409).}}  }

\begin{document}

\maketitle

\begin{abstract}
We present a simple discretization scheme for the hypersingular integral representation of the fractional Laplace operator and solver for the corresponding fractional Laplacian problem.  Through singularity subtraction, we obtain a regularized integrand that is amenable
to the trapezoidal rule with equispaced nodes, assuming a high degree of regularity in the underlying function (i.e., $u\in C^6(\R^d)$).  The resulting quadrature scheme gives a discrete
operator on a regular grid that is translation-invariant and thus can be applied quickly with the
fast Fourier transform.  For discretizations of problems related to space-fractional diffusion on
bounded domains, we observe that the underlying linear system can be efficiently solved via
preconditioned Krylov methods with a preconditioner based on the finite-difference (non-fractional)
Laplacian.  We show numerical results illustrating the error of our simple scheme as well the
efficiency of our preconditioning approach, both for the elliptic (steady-state) fractional
diffusion problem and the time-dependent problem.
\end{abstract}

\ra{1.4}

\section{Introduction}
\label{sec:intro}

Fractional powers of the Laplacian operator arise naturally in the study of anomalous diffusion,
where the fractional operator plays an analogous role to that of the standard Laplacian for ordinary
diffusion (see, e.g., the review articles by Metzler and Klafter \cite{METZLER20001,rest} and
V\'azquez \cite{vazquez}).  By replacing Brownian motion of particles with L\'evy flights
\cite{mandelbrot}, whose increments are drawn from the $\alpha$-stable L\'evy distribution for
$\alpha\in(0,2)$, we obtain a fractional diffusion equation (or fractional kinetic equation) in
terms of the fractional Laplacian operator of order $\alpha$ \cite{saichev}, defined for
sufficiently nice functions $u\from\R^d\to\R$ via the Cauchy principal value integral
\begin{align}\label{eq:fl}
\FL u(\v{x}) \equiv \text{p.v.}\int_{\R^d}  C_{\alpha,d}\left[\frac{u(\v{x})-u(\v{y})}{|\v{x}-\v{y}|^{d+\alpha}}\right]\,d\v{y}, \quad \v{x}\in\R^d,
\end{align}
with known normalizing constant $C_{\alpha,d}$ \cite{equiv}.

For a bounded domain $\Omega\subset~\R^d$ with complement $\comp{\Omega}\equiv\R^d\setminus\Omega$,
we consider fractional diffusion with homogeneous extended Dirichlet conditions given in terms of
\eqref{eq:fl} by
\begin{eqnarray}
\left\{
\begin{aligned}\label{eq:fracdiff}
\partial_t u(\v{x},t)&= -\FL u(\v{x},t) + f(\v{x},t), \qquad &&\v{x}\in\Omega, \quad &t>0,\\
u(\v{x},t) &= 0, \qquad&&\v{x}\in\comp{\Omega}, \quad &t>0,\\
u(\v{x},0) &= u_0(\v{x}), \qquad &&\v{x}\in\Omega.& 
\end{aligned}
\right.
\end{eqnarray}
Also of interest is the related elliptic problem
\begin{eqnarray}
\left\{
\begin{aligned}\label{eq:fracelliptic}
\FL u(\v{x})&=  f(\v{x}), \qquad &&\v{x}\in\Omega,\\
u(\v{x}) &= 0, \qquad &&\v{x}\in\comp{\Omega}.
\end{aligned}
\right.
\end{eqnarray}
Somewhat unintuitively, the nonlocality of \eqref{eq:fl} implies that the solutions of
\eqref{eq:fracdiff} and \eqref{eq:fracelliptic} depend on data prescribed everywhere outside
$\Omega$ \cite{Felsinger2015,ros-oton2016,delia}, though other definitions of the fractional
Laplacian on a bounded domain are also in common use \cite{vazquez}.  Further, a more general
formulation of fractional diffusion involves augmenting \eqref{eq:fracdiff} by incorporating
fractional time derivatives of Caputo or Riemann-Liouville type.  We focus in this work on the case of space-fractional diffusion and do not discuss the discretization of time-fractional differential operators, though the latter is of independent interest \cite{mclean,linxu,zeng,zengzhang,zhao}.

\subsection{Contribution}
The contribution of this paper is a simple discretization scheme for \eqref{eq:fracdiff} and
\eqref{eq:fracelliptic} on Cartesian grids, and an efficient algorithm for solving the resulting
linear systems.  The discretization generalizes easily to domains that can be represented as
occluded Cartesian grids, i.e., domains given by taking a regular grid and removing a subset of grid points and corresponding subdomains to obtain, e.g., an ``L''-shaped domain.

Our approach is based on using a Taylor expansion around each point $\v{x}$ to replace the singular
integrand in \cref{eq:fl} with a sufficiently smooth function of $\v{y}$ on all of $\R^d$ via
singularity subtraction.  The resulting integral can be easily discretized using the trapezoidal
rule on a regular grid of $N$ points, leading to a translation-invariant linear operator that can be
applied at a cost of $\O(N\log N)$ using the fast Fourier transform (FFT).  The resulting discrete
linear system approximating \eqref{eq:fracelliptic} can then be efficiently solved using standard
Krylov methods.  As $\alpha\to 2$, the resulting linear systems can exhibit the ill-conditioning
characteristic of discretizations of the Laplacian operator on a regular grid.  To circumvent this,
we develop an efficient preconditioning strategy based on the fact that our discrete fractional
Laplacian operator may be written as the sum of a standard finite-difference Laplacian and another
matrix with mostly small entries.

When the solution $u$ to $\eqref{eq:fracelliptic}$ is sufficiently smooth, standard results on
convergence of the trapezoidal rule and finite-difference operators imply that the error of our
approach for computing the fractional Laplacian at a point goes to zero as $\O(h^2)$, where $h$ is
the linear spacing between grid points, which we show in \cref{sec:discretization}.  In general,
however, the solution to the fractional Laplace problem on bounded domains is only
$\lfloor{\alpha/2}\rfloor$ times continuously differentiable \cite{bdreg}, leading to a natural
deterioration of the rate of convergence of our simple approach.

\subsection{Related work} 
A discretization scheme similar to that presented here appears in Pozrikidis
\cite{pozrikidis2016fractional}, though without discussion of accuracy or the importance of
windowing for singularity subtraction. Huang and Oberman \cite{huang-oberman,huang-oberman-new}
derive a scheme for the one-dimensional case based on singularity subtraction and finite-difference
approximation, but do not tackle the multidimensional case (see also Tian and Du \cite{tian-du}, Gao et al.\ \cite{gao}, and Duo, Van Wyk, and Zhang\cite{duo}).
Chen et al.\ \cite{Chen2014} consider a multidimensional discretization and fast preconditioners
based on multigrid, but their scheme uses the so-called ``coordinate fractional Laplacian'' that
takes a tensor product of one-dimensional operators and is not equivalent to \eqref{eq:fl} (see also
related finite-difference approaches with different operators \cite{Zhao2017,wangbasu,meer1,meer2}).

Other similar work on efficient solution of fractional Laplacian systems using fast preconditioned iterative
methods includes Pang and Sun \cite{pangmg} and Wang and collaborators \cite{wangsuperfast,fuwang,fu,fu2019}.  While limited to one spatial dimension, this work also exploits the structure of the discrete operator for fast matrix-vector products and preconditioned solves, and the latter line of work includes treatment of time-fractional operators.

Another family of approaches on discretizing the fractional Laplacian operator is based on finite elements \cite{acosta-fractional,ag,ag2,blp,bonito-diffusion,tiandu,petr}.  Compared to our scheme, such approaches are typically more amenable to general geometries (as is typical for finite elements) but are also more involved.  Other notable schemes for discretizing the fractional Laplacian based on different
ideas include work based on the Caffarelli-Silvestre extension \cite{cs-extension,Nochetto2015,hulili}, spectral approaches \cite{zayernouri,mao,bruno,xudarve}, and hybrid schemes \cite{agsiam}.  General references for fractional Laplacians on bounded domains include, e.g., Ros-Oton
\cite{ros-oton2016}, D'Elia and Gunzburger \cite{delia}, Felsinger \cite{Felsinger2015}, and Lischke
et al.\ \cite{lpg}.

\section{Spatial discretization of the fractional Laplacian}
\label{sec:discretization}

To begin, we outline our scheme for discretization of \eqref{eq:fl} in the one-dimensional case
where the function $u$ vanishes outside of some interval.  Following that, we give more details in
our discussion of the multidimensional case.

\subsection{Singularity subtraction in one dimension}
Concretely, consider the task of approximating the principal value integral
\begin{align}\label{eq:1dintegral}
\FL u(x)= \text{p.v.}\int_{-\infty}^\infty C_{\alpha,1}\left[\frac{u(x)-u(y)}{|x-y|^{1+\alpha}}\right]\,dy,
\end{align}
where $u(y) = 0$ for $|y|>1$. For $\alpha>1$, this integral is hypersingular due to the high-order
pole at $x=y$, which generally leads to large inaccuracies when simple quadrature schemes are
applied directly to \eqref{eq:1dintegral}.  Therefore, we proceed by regularizing the integrand to
remove the singularity and obtain an integral for which simple quadratures are accurate.

Assuming that the function $u$ is sufficiently smooth, we may write a Taylor series expansion about
the point $x$ to obtain
\begin{align}\label{eq:oddpart}
u(y) &= u(x) + \frac{1}{2}u''(x)(y-x)^2 + u_\text{odd}(y) + R_4(y),
\end{align}
where the smooth remainder $R_4(y) = O(|y-x|^4)$ as $y\to x$.  For brevity, we have grouped terms
that are odd about $x$ into $u_\text{odd}$,
as they will not play an explicit role in what follows.

Our regularization strategy is singularity subtraction based on adding and subtracting a calculable
integral that matches terms in the Taylor series.  Suppose $w$ is a sufficiently smooth windowing
function with compact support such that $w(0) = 1$ and $w(y)=w(-y)$.  Then we may write
\begin{align}\label{eq:twointegrals}
\FL u(x) = &C_{\alpha,1}\int_{-\infty}^\infty \frac{u(x)-u(y) + w(x-y)[\frac{1}{2}u''(x)(x-y)^2-u_\text{odd}(y)]}{|x-y|^{1+\alpha}}\,dy\\
&-C_{\alpha,1}\int_{-\infty}^\infty \frac{w(x-y)[\frac{1}{2}u''(x)(x-y)^2-u_\text{odd}(y)]}{|x-y|^{1+\alpha}}\,dy\nonumber\\
&\equiv \I + \II,\nonumber
\end{align}
where we define $\I$ to be the first integral and $\II$ to be the second.
By construction, $\I$ is no longer hypersingular, as we see from \eqref{eq:oddpart} that the integrand can be equivalently written
\begin{align*}\label{eq:integrand}
 \frac{[w(x-y)-1][\frac{1}{2}u''(x)(x-y)^2-u_\text{odd}(y)]+R_4(y)}{|x-y|^{1+\alpha}}.
\end{align*}
By our smoothness assumptions on $u$ and $w$, as $y\to x$ this integrand decays and is
continuously differentiable with a second derivative that is integrable.  This implies
that the standard trapezoidal rule would exhibit second-order convergence when applied to $\I$; see
Cruz-Uribe and Neugebauer \cite{trapz}.  Of course, this requires knowledge of $u''(x)$ and
$u_\text{odd}$ in general, which we do not assume.  In the context of discretization of the integral
using a uniform grid, however, the situation simplifies.

\subsection{The first integral in one dimension}
Consider discretizing $\I$ using the trapezoidal rule on a one-dimensional lattice $\{y_j\}_{j\in\Z}
= \{jh\}_{j\in\Z}$ and take $x = y_i$ to be one of the lattice points.  Without loss of generality,
we may shift the domain such that $x=y_0 = 0$.  This discretization yields the second-order accurate
approximation
\begin{align*}
\I &\approx C_{\alpha,1}h\sum_{j\ne 0}\left[\frac{u(0)-u(y_j) + w(y_j)[\frac{1}{2}u''(0)(y_j)^2-u_\text{odd}(y_j)]}{|y_j|^{1+\alpha}}\right]\\
&=C_{\alpha,1}h\left[\sum_{j\ne0} \frac{u(0)}{|y_j|^{1+\alpha}}  - \sum_{j\ne0} \frac{u(y_j)}{|y_j|^{1+\alpha}} + \frac{u''(0)}{2}\sum_{j\ne0}\frac{w(y_j)}{|y_j|^{\alpha-1}} - \sum_{j\ne 0} \frac{w(y_j) u_\text{odd}(y_j)}{|y_j|^{1+\alpha}} \right]\\
&= C_{\alpha,1}h\left[A_1u(0)  - \sum_{j\ne0} \frac{u(y_j)}{|y_j|^{1+\alpha}} + A_2u''(0)\right],
\end{align*}
where $A_1 = \sum_{j\ne 0} |y_j|^{-(1+\alpha)}$ and $A_2=\frac{1}{2}\sum_{j\ne 0} w(y_j)|y_j|^{1-\alpha}$ are
constants independent of $x$ and $j$ and the last sum in the second line is identically zero due to
oddness considerations.  We note that the sum remaining on the final line is over a finite range, as
$u$ is compactly supported.  Since $u$ is assumed to be smooth enough, we replace $u''(0)$ with the
finite-difference approximation
\begin{align*}
u''(0) \approx L_\text{FD}u(0) \equiv \frac{u(h) - 2 u(0) + u(-h)}{h^2},
\end{align*}
which gives our final approximation for $\I$,
\begin{align}
\I \approx C_{\alpha,1}h\left[A_1u(0)  - \sum_{j\ne0} \frac{u(y_j)}{|y_j|^{1+\alpha}} + A_2L_\text{FD}u(0)\right].
\end{align}

\subsection{The second integral in one dimension and final quadrature}
Having established a method for approximating the integral $\I$ in \eqref{eq:twointegrals}, we turn
to $\II$.  Again using oddness considerations, we see that the contribution from $u_\text{odd}$
vanishes such that
\begin{align*}
\II &= -\frac{C_{\alpha,1}u''(0)}{2}\int_{-\infty}^\infty \frac{w(y)}{|y|^{\alpha-1}}\,dy = C_{\alpha,1}hA_3u''(0),
\end{align*}
where the constant $A_3$ given by
\begin{align*}
A_3 = -\frac{1}{2h}\int_{-\infty}^\infty\frac{w(y)}{|y|^{\alpha-1}}\,dy
\end{align*}
 is well-defined (since $w$ is compactly supported) and we again take $x=0$ for convenience.  We
 once again replace the second derivative $u''(0)$ with its finite-difference approximation to
 obtain $ \II \approx C_{\alpha,1}hA_3L_\text{FD}u(0).$ Combining this with our quadrature for $\I$
 gives our approximation for $\FL u(0)$,
 \begin{align*}
 \FL u(0) &\approx C_{\alpha,1}h\left[A_1u(0)  - \sum_{j\ne0} \frac{u(y_j)}{|y_j|^{1+\alpha}} + (A_2+A_3)L_\text{FD}u(0)\right],
 \end{align*}
 which applies equally well not only to $x=0$ but in general to $x=y_i$ for any grid point $y_i$, i.e.,
 \begin{align}
\FL u(y_i) &\approx C_{\alpha,1}h\left[A_1u(y_i)  - \sum_{j\ne i} \frac{u(y_j)}{|y_i-y_j|^{1+\alpha}} + (A_2+A_3)L_\text{FD}u(y_i)\right].
 \end{align}
This is our final quadrature for the fractional Laplacian in one dimension.

\subsection{Singularity subtraction in higher dimensions}
We turn now to the multidimensional integral, i.e., \eqref{eq:fl} with $d=2$ or $d=3$.  Once again
we will assume that the function $u$ is compactly supported and sufficiently smooth, as we will make
explicit.  Our basic strategy is the same as in one dimension.

\begin{lemma}\label{lem:taylor}
Suppose that $u\in C^k(\R^d)$ and let $w\in C^p(\R)$ be a windowing function symmetric about $z=0$
such that $1-w(z) = \O(|z|^r)$ as $z\to0$.  Let the third-order Taylor approximation of $u$ about
the point $\v{x}\in\R^d$ be given in multi-index notation by
\begin{align}\label{eq:taylormulti}
u(\v{y}) &= \sum_{|\beta|\le 3} \frac{D^\beta u(\v{x})}{\beta!}(\v{y}-\v{x})^\beta + \sum_{|\tilde\beta| = 4} R_{\tilde\beta}(\v{y})(\v{y}-\v{x})^{\tilde\beta},
\end{align}
where the remainder is given in explicit form as
\begin{align*}
R_{\tilde\beta}(\v{y}) \equiv \frac{|\tilde\beta|}{\tilde\beta !}\int_0^1 (1-t)^{|\tilde\beta|-1}D^{\tilde\beta} u(\v{x} + t(\v{y}-\v{x}))\,dt.
\end{align*}
Then, defining the function
\begin{align}
\tilde u(\v{y}) &\equiv u(\v{y}) - u(\v{x}) - w(|\v{x}-\v{y}|)\sum_{1\le|\beta|\le 3} \frac{D^\beta u(\v{x})}{\beta!}(\v{y}-\v{x})^\beta,
\end{align}
we have that $\tilde u\in C^s(\R^d)$ and $D^\beta\tilde u(\v{y}) = \O(|\v{y}-\v{x}|^{t-|\beta|})$  as $\v{y}\to\v{x}$ for $s=\min(k-4,p)$, $t=\min(1+r,4)$, and $0\le|\beta|\le \min(s,t)$.
\end{lemma}
\begin{proof}
It is clear that
\begin{align*}
\tilde u(\v{y}) &= (1 - w(|\v{x}-\v{y}|)) \sum_{1\le|\beta|\le 3} \frac{D^\beta u(\v{x})}{\beta!}(\v{y}-\v{x})^\beta + \sum_{|\tilde\beta| = 4} R_{\tilde\beta}(\v{y})(\v{y}-\v{x})^{\tilde\beta}.
\end{align*}
By inspection, the order of differentiability of $\tilde u(\v{y})$ is limited by that of
$w(|\v{x}-\v{y}|)$ and of $R_{\tilde\beta}(\v{y})$.  Given the explicit form of
$R_{\tilde\beta}(\v{y})$, it is at least in $C^{k-4}(\R^d)$ as a function of $\v{y}$, whereas $w \in
C^p(\R)$ by assumption.  Further, $\tilde u(\v{y}) = \O(|\v{y}-\v{x}|^t)$ for $t=\min(1+r,4)$, since
the first summand is $\O(|\v{y}-\v{x}|^{1+r})$ and the second summand is at least
$\O(|\v{y}-\v{x}|^{4})$.  Explicit term-by-term differentiation of $\tilde u(\v{y})$ with the
product rule concludes the proof.
\end{proof}

By subtracting off the windowed multivariate Taylor series we obtain an integral that is no longer
hypersingular.  In particular, we write
\begin{align}\label{eq:twointegralsmultid}
\FL u(\v{x}) = &C_{\alpha,d}\int_{\R^d} \frac{u(\v{x}) - u(\v{y}) + w(|\v{x}-\v{y}|)\sum_{1\le|\beta|\le 3} \frac{D^\beta u(\v{x})}{\beta!}(\v{y}-\v{x})^\beta}{|\v{x}-\v{y}|^{d+\alpha}}\,d\v{y}\\
&-C_{\alpha,d}\int_{\R^d} \frac{w(|\v{x}-\v{y}|)\sum_{1\le|\beta|\le 3} \frac{D^\beta u(\v{x})}{\beta!}(\v{y}-\v{x})^\beta}{|\v{x}-\v{y}|^{d+\alpha}}\,d\v{y}\nonumber\\
&\equiv \Id + \IId,\nonumber
\end{align}
where we define $\Id$ to be the first integral and $\IId$ to be the second.

\subsection{The first integral in higher dimensions}
To numerically approximate $\Id$ we use a quadrature rule on a uniform lattice
$\{\v{y}_\v{j}\}_{\v{j}\in\Z^d} = \{\v{j}h\}_{\v{j}\in\Z^d}$.  We assume the lattice is constructed
such that the point $\v{x}$ coincides with with some lattice point $\v{y}_\v{i}$, which we take to
be $\v{x}=\v{y}_\v{0}=\v{0}$ without loss of generality.

Replacing the integral with a weighted sum over the lattice, we obtain
\begin{align*}
\Id &\approx C_{\alpha,d}h^d\sum_{\v{j}\ne \v{0}}\frac{u(\v{y}_\v{i}) - u(\v{y}_\v{j}) + w(|\v{y}_\v{i}-\v{y}_\v{j}|)\sum_{1\le|\beta|\le 3} \frac{D^\beta u(\v{y}_\v{i})}{\beta!}(\v{y}_\v{j}-\v{y}_\v{i})^\beta}{|\v{y}_\v{i}-\v{y}_\v{j}|^{d+\alpha}}\\
&= C_{\alpha,d}h^d\sum_{\v{j}\ne \v{0}}\frac{u(\v{0})-u(\v{y}_\v{j}) + w(|\v{y}_\v{j}|)\sum_{1\le|\beta|\le 3} \frac{D^\beta u(\v{0})}{\beta!}(\v{y}_\v{j})^\beta}{|\v{y}_\v{j}|^{d+\alpha}},
\end{align*}
which we note does not include a term for $\v{j}=\v{0}$.  This corresponds to the standard
trapezoidal rule for $d=2$ and the punctured trapezoidal rule for $d=3$, though more involved
quadrature corrections may be used (see, e.g., Marin, Runborg and Tornberg \cite{marin}).  Assuming
$w$ is symmetric about the origin, we see that for many values of the multi-index $\beta$ the
corresponding summand vanishes due to oddness considerations.  Taking these symmetries into account,
we let $\v{e}_1^T\v{y}_\v{j}$ denote the first coordinate of $\v{y}_\v{j}$ and observe that
\begin{align*}
\sum_{\v{j}\ne \v{0}}\sum_{1\le|\beta|\le 3} \frac{ w(|\v{y}_\v{j}|) \frac{D^\beta u(\v{0})}{\beta!}(\v{y}_\v{j})^\beta	}{|\v{y}_\v{j}|^{d+\alpha}} &=\frac{\Delta u(\v{0})}{2}\sum_{\v{j}\ne \v{0}} \frac{w(|\v{y}_\v{j}|)(\v{e}_1^T\v{y}_\v{j})^2}{|\v{y}_\v{j}|^{d+\alpha}},
\end{align*}
which we plug back into our quadrature scheme to obtain
\begin{align}
\Id &\approx C_{\alpha,d}h^d\sum_{\v{j}\ne \v{0}}\frac{u(\v{0})-u(\v{y}_\v{j}) + \frac{\Delta u(\v{0})}{2} w(|\v{y}_\v{j}|)(\v{e}_1^T\v{y}_\v{j})^2}{|\v{y}_\v{j}|^{d+\alpha}}\nonumber\\
&=  C_{\alpha,d}h^d \left[\left(\sum_{\v{j}\ne \v{0}}\frac{1}{|\v{y}_\v{j}|^{d+\alpha}}\right) u(\v{0}) - \sum_{\v{j}\ne \v{0}}\frac{u(\v{y}_\v{j}) }{|\v{y}_\v{j}|^{d+\alpha}}   + \left(\frac{1}{2}\sum_{\v{j}\ne \v{0}} \frac{w(|\v{y}_\v{j}|)(\v{e}_1^T\v{y}_\v{j})^2}{|\v{y}_\v{j}|^{d+\alpha}}\right)\Delta u(\v{0})\right]\nonumber\\
&\nonumber \equiv C_{\alpha,d}h^d \left[A_{1,d} u(\v{0}) - \sum_{\v{j}\ne \v{0}}\frac{u(\v{y}_\v{j}) }{|\v{y}_\v{j}|^{d+\alpha}}   + A_{2,d}\Delta u(\v{0})\right],
\end{align}
with correspondingly defined constants
\begin{align}\label{eq:int1constants}
A_{1,d} \equiv \left(\sum_{\v{j}\ne \v{0}}\frac{1}{|\v{y}_\v{j}|^{d+\alpha}}\right), \quad A_{2,d}\equiv \left(\frac{1}{2}\sum_{\v{j}\ne \v{0}} \frac{w(|\v{y}_\v{j}|)(\v{e}_1^T\v{y}_\v{j})^2}{|\v{y}_\v{j}|^{d+\alpha}}\right).
\end{align}

\begin{theorem}\label{thm:thetheorem}
Suppose the same setup as \cref{lem:taylor} with $k=6$, $p=3$, and $r=3$ such that $t=4$ and $s=2$.
Assume further $u$ and $w$ are compactly supported with $0\le w(z)\le 1$ for all $z$.  Then the
above approximation for $\Id$ is second-order accurate.  That is,
\begin{align*}
&C_{\alpha,d}\int_{\R^d} \frac{u(\v{0}) - u(\v{y}) + w(|\v{y}|)\sum_{1\le|\beta|\le 3} \frac{D^\beta u(\v{0})}{\beta!}(\v{y})^\beta}{|\v{y}|^{d+\alpha}}\,d\v{y}\\
&=
C_{\alpha,d}h^d \left[A_{1,d} u(\v{0}) - \sum_{\v{j}\ne \v{0}}\frac{u(\v{y}_\v{j}) }{|\v{y}_\v{j}|^{d+\alpha}}   + A_{2,d}\Delta u(\v{0})\right] + \O(h^2),
\end{align*}
with $A_{1,d}$ and $A_{2,d}$ as in \eqref{eq:int1constants}.
\end{theorem}
\begin{proof}
The described approximation is numerically equivalent to the (punctured) trapezoidal rule, so this
amouts to bounding the error of the trapezoidal rule applied in $d$ dimensions with integrand
$\tilde u(\v{y}) / |\v{y}|^{d+\alpha}$, where $\tilde u(\v{y})$ is as in \cref{lem:taylor} with
$\v{x}=\v{0}$.  Letting $R>h$ be such that both $u(\v{y})=0$ and $w(|\v{y}|)=0$ for $|\v{y}|> R$, we
proceed by breaking the integral into three contributions: one for the subdomain $B_h\equiv
[-h,h]^d$ ``near'' the singularity, one for the ``mid-range'' subdomain $B_R\setminus B_h
\equiv[-R,R]^d\setminus[-h,h]^d$, and one for the ``far'' subdomain $\R^d\setminus B_R$.  We write
\begin{align*}
\int_{\R^d} \frac{\tilde u(\v{y})}{|\v{y}|^{d+\alpha}}\,d\v{y} &= \int_{B_h} \frac{\tilde u(\v{y})}{|\v{y}|^{d+\alpha}}\,d\v{y} + \int_{B_R\setminus B_h} \frac{\tilde u(\v{y})}{|\v{y}|^{d+\alpha}}\,d\v{y}
+\int_{\R^d\setminus B_R} \frac{\tilde u(\v{y})}{|\v{y}|^{d+\alpha}}\,d\v{y},
\end{align*}
each piece of which we analyze separately.

Near the singularity, we see due to symmetry considerations that
\begin{align*}
&\int_{B_h} \frac{\tilde u(\v{y})}{|\v{y}|^{d+\alpha}}\,d\v{y} = \sum_{1\le|\beta|\le 3} \frac{D^\beta u(\v{0})}{\beta!}\int_{B_h} \frac{(1 - w(|\v{y}|))(\v{y})^\beta}{|\v{y}|^{d+\alpha}}\,d\v{y} + \sum_{|\tilde\beta| = 4} \int_{B_h}\frac{R_{\tilde\beta}(\v{y})(\v{y})^{\tilde\beta}}{|\v{y}|^{d+\alpha}}\,d\v{y}\\ &=  \frac{\Delta u(\v{0})}{2}\int_{B_h} \frac{(1 - w(|\v{y}|))(\v{e}_1^T\v{y})^2}{|\v{y}|^{d+\alpha}}\,d\v{y} + \sum_{|\tilde\beta| = 4} \int_{B_h}\frac{R_{\tilde\beta}(\v{y})(\v{y})^{\tilde\beta}}{|\v{y}|^{d+\alpha}}\,d\v{y}= O(h^{4-\alpha}),
\end{align*}
where under our assumptions the integrands are both $O(h^{4-d-\alpha})$ and $\Delta u$ is bounded.
Since $\tilde u(\v{y}) = O(|\v{y}|^t)$, we see $\frac{\tilde u(\v{y})}{|\v{y}|^{d+\alpha}} =
O(|\v{y}|^{t-d-\alpha})$, which implies that the corresponding (punctured) trapezoidal rule
approximation to the integral is $O(h^{t-\alpha})$, since we gain a factor of $h^d$ due to the
quadrature weights.  Therefore, the contribution to the error from the integral over the near
subdomain is $O(h^{4-\alpha}) = O(h^2)$, since $\alpha\in(0,2)$.

In the mid-range subdomain, we explicitly use the composite nature of the trapezoidal rule to write 
\begin{align*}
\int_{B_R\setminus B_h} \frac{\tilde u(\v{y})}{|\v{y}|^{d+\alpha}}\,d\v{y} &= \sum_{\ell} \int_{\Omega_\ell} \frac{\tilde u(\v{y})}{|\v{y}|^{d+\alpha}}\,d\v{y},
\end{align*}
and then consider the error of the trapezoidal rule in approximating the integral over each
$\Omega_\ell$ separately, where the square/cubic subdomains $\{\Omega_\ell\}$ in the trapezoidal
rule are pairwise disjoint and are such that $\bigcup_{\ell}\Omega_\ell = B_R\setminus B_h$.  Since
we are away from the origin, on each subdomain $\Omega_\ell$ the integrand
$\phi(\v{y})\equiv\frac{\tilde u(\v{y})}{|\v{y}|^{d+\alpha}}$ is in $C^2(\Omega_\ell)$ which means
the standard error bound for the trapezoidal rule on $\Omega_\ell$ gives an error contribution of no
more than $C h^{d+2}\sum_{|\beta|=2} \|D^\beta \phi\|_{L_\infty(\Omega_\ell)}$ for some constant $C$
independent of $h$.  However, the term $\|D^\beta \phi\|_{L_\infty(\Omega_\ell)}$ does depend on
$h$.  Since $D^\beta\tilde u(\v{y}) = O(|\v{y}|^{t-|\beta|})$ from \cref{lem:taylor}, the product
rule gives $D^\beta \phi(\v{y}) = O(1+|\v{y}|^{t-|\beta|-d-\alpha})$.  With this we can bound the
total error on $\R^d\setminus B_h$ as
\begin{align*}
\sum_\ell C h^{d+2}\sum_{|\beta|=2} \|D^\beta \phi\|_{L_\infty(\Omega_\ell)} &\le C'h^{d+2}\sum_\ell\|1+|\v{y}|^{t-2-d-\alpha}\|_{L_\infty(\Omega_\ell)}\\
&\le C'' h^2\left(1 + \int_0^R r^{t-3-\alpha}\,dr\right) = C''' h^2,
\end{align*}
where we have bounded 
\begin{align*}
h^d\sum_\ell \|1+|\v{y}|^{t-2-d-\alpha}\|_{L_\infty(\Omega_\ell)} \le c\int_{B_R}(1 + |\v{y}|^{t-2-d-\alpha})\,d\v{y} + c'
\end{align*}
(up to some geometry-dependent factors that are independent of $h$) due to concavity of the summand.
Therefore, the error contribution from the mid-range subdomain is $O(h^2)$.

Finally, for the far subdomain, we observe that the integrand is in $C^2(\R^d\setminus B_R)$ and its
smoothness is independent of $h$ in this region, so the standard composite trapezoidal error bound
of $O(h^2)$ applies.  Therefore, the overall error is $O(h^2)$.
\end{proof}
\begin{remark}
Being based on singularity subtraction via Taylor series expansion, the theoretical results in \cref{lem:taylor} and \cref{thm:thetheorem} apply directly only for relatively smooth functions $u$.  As discussed, however, it is known that in the general case solutions to \eqref{eq:fracelliptic} exhibit only mild H\"older regularity on the whole space but typically better regularity on $\Omega$ (i.e., $u\in C^{0,\alpha/2}(\R^d)$ but $u$ is more regular than $f$ on $\Omega$)\cite{bdreg}.  This lack of regularity across the boundary of $\Omega$ substantially complicates error analysis of any translation-invariant numerical approach such as is presented here.
\end{remark}

While smoothness is not generally a property of solutions to \eqref{eq:fracelliptic}, examples can be concocted.
	For example, inside the unit ball $B\equiv\{\v{x} \mid |\v{x}|^2 \le 1\} \subset \R^d$ one family of smooth solutions is given by observing that for $q>0$ and $s\in(0, 1)$ we have
	\begin{align*}
	(-\Delta)^{-s} \left[ (1-|\v{x}|^2)^q_+\right] &= K\times{}_2{F}_1\left(\frac{d}{2}- s, -q-s; \frac{d}{2}; |\v{x}|^2\right), \; |\v{x}| \le 1
	\end{align*}
	for known constant $K$\cite[eq. 9]{barenblatt}, where ${}_2{F}_1$ is the Gauss hypergeometric function \cite{abramowitz1964handbook}.  Applying the negative Laplacian to either side and letting $s=1-\alpha/2$ we see
	\begin{align*}
	\FL \left[ (1-|\v{x}|^2)^q_+\right] &= -\Delta\left[K\times{}_2{F}_1\left(\frac{d+\alpha}{2} - 1, \frac{\alpha}{2}-q-1; \frac{d}{2}; |\v{x}|^2\right)\right], \; |\v{x}| \le 1.
	\end{align*}
	This gives a family of smooth solutions to \eqref{eq:fracelliptic} on $B$, and related formulas can be used to obtain $\FL \left[ (1-|\v{x}|^2)^q_+\right]$ for $|\v{x}|>1$ (and thus to extend the problem domain beyond $B$).  Beyond such examples, the theoretical accuracy of \cref{thm:thetheorem} is chiefly useful when studying the fractional Laplacian \emph{forward operator} applied to smooth functions.  That said, in \cref{sec:results} we empirically observe linear convergence of the solution to \eqref{eq:fracelliptic} for $\alpha > 1$.

\subsection{The second integral in higher dimensions and final quadrature}
We now consider the second integral $\IId$ in \eqref{eq:twointegralsmultid}.  Assuming without loss
of generality that $\v{x} =0$ and using symmetry and oddness considerations as before, we see that
\begin{align*}
\IId &= -C_{\alpha,d}\int_{\R^d} \sum_{1\le|\beta|\le 3} \frac{w(|\v{y}|)\frac{D^\beta
    u(\v{0})}{\beta!}(\v{y})^\beta}{|\v{y}|^{d+\alpha}}\,d\v{y}= -\frac{C_{\alpha,d}\Delta
  u(0)}{2}\int_{\R^d} \frac{w(|\v{y}|)(\v{e}_1^T\v{y})^2}{|\v{y}|^{d+\alpha}}\,d\v{y}.
\end{align*}
Defining the constant
\begin{align}\label{eq:a3}
A_{3,d} &\equiv -\frac{h^{-d}}{2}\int_{\R^d} \frac{w(|\v{y}|)(\v{e}_1^T\v{y})^2}{|\v{y}|^{d+\alpha}}\,d\v{y}
\end{align}
and combining this with our quadrature for $\Id$ gives
\begin{align*}
\FL u(\v{0}) \approx C_{\alpha,d}h^d \left[A_{1,d} u(\v{0}) - \sum_{\v{j}\ne \v{0}}\frac{u(\v{y}_\v{j}) }{|\v{y}_\v{j}|^{d+\alpha}}  + (A_{2,d} + A_{3,d})\Delta u(\v{0})\right] 
\end{align*}
or, more generally,
\begin{align*}
\FL u(\v{y}_\v{i}) \approx C_{\alpha,d}h^d \left[A_{1,d} u(\v{y}_\v{i}) - \sum_{\v{j}\ne \v{i}}\frac{u(\v{y}_\v{j}) }{|\v{y}_\v{i}-\v{y}_\v{j}|^{d+\alpha}}   + (A_{2,d} + A_{3,d})\Delta u(\v{y}_\v{i})\right]. 
\end{align*}
Of course, as written this approximation requires second derivative information in the form of
$\Delta u(\v{y}_\v{i})$.  For smooth $u$, however, we may replace this with a finite-difference
stencil involving the neighbors of $\v{y}_{i}$ in the lattice,
\begin{align*}
\Delta u(\v{y}_\v{i}) \approx L_\text{FD}u(\v{y}_\v{i}) \equiv \frac{1}{h^2}\left(\sum_{\|\v{i}-\v{j}\|_1 = 1} u(\v{y}_\v{j})  - 2^d u(\v{y}_\v{i})\right),
\end{align*}
just as in the one-dimensional case.

\subsection{Summary of quadrature for fractional Laplacian}
We briefly summarize our complete approach for discretizing the fractional Laplacian applied to a
function $u$.  First, we regularize the integrand of \eqref{eq:fl} by adding to the numerator a
windowed Taylor series approximation of $u$ about $\v{x}$ with window function $w$ to obtain $\Id$
in \eqref{eq:twointegralsmultid}.  This gives an integral that is nice enough to admit
discretization with the trapezoidal rule or related schemes.  Then, by exploiting symmetries of the
problem, we rewrite the discretization in terms of the constants $A_{1,d}$ and $A_{2,d}$ in
\eqref{eq:int1constants}, which do not depend on $u$.  Finally, we derive an expression for the
correction term $\IId$ in terms of another constant $A_{3,d}$ given in \eqref{eq:a3}, which when
combined with $\Id$ and a finite-difference stencil approximation gives a nice expression for $\FL
u(\v{y}_\v{i})$ as a linear function of $u$ evaluated on a regular grid.

A few details of the procedure remain to be discussed.  First, there are a number of possibilities
for the windowing function $w$.  In this paper, we use the piecewise-polynomial window
\begin{align}\label{eq:window}
w(r) = W_\delta(r) \equiv  \left\{\begin{array}{ll} 1 - 35\left(\frac{r}{\delta}\right)^4 + 84\left(\frac{r}{\delta}\right)^5 - 70\left(\frac{r}{\delta}\right)^6 + 20\left(\frac{r}{\delta}\right)^7, & r < \delta, \\ 0, & \text{else}. \end{array} \right.
\end{align}
Of course, this is by no means the only sufficiently smooth choice.  Further, we note that the
requirement that $w$ be compactly supported can be relaxed so long as $W$ decays sufficiently
quickly as $r\to\infty$ such that the necessary integrals and sums may be computed.

On that note, we also must still compute the constants $A_{1,d}$, $A_{2,d}$, and $A_{3,d}$.  For our
choice of polynomial window, the integral defining $A_{3,d}$ can be computed explicitly; for other
choices the integral may be numerically computed to high precision offline using, e.g., adaptive
quadrature in MATLAB.  For compactly supported $w$, the sum defining $A_{2,d}$ has a finite number of
nonzero terms and is easily computable.  Finally, the infinite lattice sum $A_{1,d}$ is given in
terms of the Riemann zeta function for $d=1$ and may otherwise be well-approximated using far-field
compression techniques related to the fast multipole method (FMM) \cite{fastmultipole,kifmm}.  We
use Chebyshev polynomials for far-field compression in the vein of Fong and Darve \cite{bbfmm},
though we do not require the full FMM machinery as we are interested only in the lattice sum and not
a full approximate operator.

We remark that the analysis of this section gives a bound for the ``apply error'' when the approximate operator is applied to an appropriately smooth function.  While we use standard regularity assumptions to prove convergence of the finite-difference quadrature approximation, such regularity does not hold in general for solutions to \eqref{eq:fracelliptic}, particularly near the boundary $\partial\Omega$ \cite{rostongen}.  Thus, these results do not apply directly to the ``solve error'' (error in approximating $u$), and in practice we expect lower rates of convergence for the solve error, as we explore numerically in \cref{sec:results}.

\section{Solving the fractional differential equations on a bounded domain}
Having developed our trapezoidal rule scheme for evaluating \eqref{eq:fl} given $u$, we turn now to
the fractional differential equations \eqref{eq:fracdiff} and \eqref{eq:fracelliptic} concerning
fractional diffusion on a bounded domain $\Omega$ with homogeneous extended Dirichlet conditions.
We focus on the case $\Omega=[0,1]^d$ for ease of exposition.

\subsection{The elliptic case: steady-state fractional diffusion}\label{sec:elliptic-disc}
To solve the elliptic problem \eqref{eq:fracelliptic}, we discretize $\Omega$ using a regular grid
of $N=(n-1)^d$ points $\{\v{y}_\v{j}\}$ with linear spacing $h=\frac{1}{n+1}$, where $\v{j} = (j_1,
\dots, j_d)$ and $\v{y}_\v{j} = h \v{j}.$ For notational convenience, we define the index set
$\J\equiv [n]^d\subset \Z^d$.  Then, replacing the fractional Laplacian with our quadrature-based
approximation gives
\begin{align}\label{eq:linearelliptic}
  C_{\alpha,d}h^d \left[A_{1,d} u_\v{i} - \sum_{\substack{\v{j}\in\J\\ \v{j}\ne\{\v{i}\}}}\frac{u_\v{j} }{|\v{y}_\v{i}-\v{y}_\v{j}|^{d+\alpha}}   + (A_{2,d} + A_{3,d})L_\text{FD}u_\v{i}\right] &= f(\v{y}_\v{i})\quad \forall \v{i}\in\J,
\end{align}
which is a linear system to be solved for the variables $\{u_\v{j}\} \approx \{u(\v{y}_\v{j})\}$.
We remark that the ``boundary conditions'' affect the system in two ways.  First, the center sum has
been reduced from an infinite number of terms (in general) to a more manageable finite sum.  Second,
evaluating the finite-difference stencil $L_\text{FD}$ for $\v{i}$ near the boundary of the domain
will require the prescribed value of $u(\v{y})$ on the boundary, as in the standard (non-fractional)
case.

We write \eqref{eq:linearelliptic} in matrix form as
\begin{align}\label{eq:matelliptic}
  \m{M} \v{u} = \v{f},
\end{align}
where now $\v{u}\in\R^N$ and $\v{f}\in\R^N$ are vectors with corresponding entries $\{u_\v{j}\}$ and
$\{f(\v{y}_\v{j})\}$ and $\m{M}\in\R^{N\times N}$ contains the coefficients implied by
\eqref{eq:linearelliptic}.

\subsubsection*{Forward operator and application with FFT}
By construction, the approximate fractional Laplacian operator involved in \eqref{eq:linearelliptic}
is translation-invariant, which means that the matrix $\m{M}$ is block Toeplitz with Toeplitz blocks
(BTTB) under any natural ordering of the unknowns.  As is well known, this in turn implies that
$\m{M}$ may be applied efficiently using the FFT at a cost of $O(N \log N)$ FLOPs per application
and stored with storage cost $O(N)$.

Further, investigation of the constants $A_{1,d}$, $A_{2,d}$ and $A_{3,d}$ reveal that $\m{M}$ is
symmetric positive definite.  When coupled with the previous observatiton, this leads naturally to
the use of the conjugate gradient method (CG) \cite{cg} or related iterative methods for solving
\eqref{eq:matelliptic}.  However, while the FFT ensures low complexity per iteration, the number of
iterations required to achieve a specified iteration can be large unless an effective preconditioner
is used. This is of particular concern as $\alpha \to 2$, whereupon we recover the standard
(ill-conditioned) Laplacian.

\subsubsection*{Preconditioning: Laplacian pattern and fast Poisson solver}

To construct an efficient preconditioner for \eqref{eq:matelliptic}, we observe that $\m{M}$ may be
decomposed as the sum of two matrices $\m{M}=C_{\alpha,d}h^d(\m{K}+\m{L})$, where
\begin{align*}
\m{K}_{\v{i}\v{j}} &= \left\{\begin{array}{cc} -\frac{1}{|\v{y}_\v{i} -\v{y}_\v{j}|^{d+\alpha}},& \v{i}\ne \v{j},\\ A_{1,d}, & \v{i} = \v{j}, \end{array} \right.\quad \text{and}\quad
\m{L}_{\v{i}\v{j}} = \left\{\begin{array}{rr} \frac{(A_{2,d}+A_{3,d})}{h^2},& \|\v{i}-\v{j}\|_1 = 1,\\ -\frac{2^d(A_{2,d}+A_{3,d})}{h^2}, & \v{i} = \v{j},\\
0, & \text{else}, \end{array} \right.
\end{align*}
and we note that $A_{2,d} + A_{3,d} < 0$.  The sparse matrix $\m{L}$ is (up to a proportionality
constant) the typical finite-difference approximation of the negative Laplacian, whereas the matrix
$\m{K}$ has entries that quickly decay away from $\v{i}=\v{j}$, particularly for larger $\alpha$.  This motivates using $\m{L}$ itself as a preconditioner when using CG to solve \eqref{eq:matelliptic}.  Because $\m{L}$ is
effectively a finite-difference discretization of Poisson's equation on a regular grid with
homogeneous Dirichlet boundary conditions, application of $\m{L}^{-1}$ may be accomplished with the
FFT at a cost of $O(N\log N)$ using typical fast Poisson solver techniques \cite[Chapter
  12]{iserles1996first}.  For non-rectangular domains, the FFT-based approach is no longer feasible,
but the same preconditioner can be used with, e.g., nested dissection \cite{nd} or related
methods.

We remark that other choices of preconditioner are possible.  For example, rather than using
$\m{L}^{-1}$ as our preconditioner we could instead use $\widetilde{\m{M}}^{-1}$, where
$\widetilde{\m{M}}_{\v{i}\v{j}} = \m{M}_{\v{i}\v{j}}$ if $\m{L}_{\v{i}\v{j}}\ne0$ and zero
otherwise.  Preliminary experiments with this approach (not shown) did not seem to show measurable
benefit.

\subsection{The time-dependent case: time-dependent fractional diffusion}\label{sec:time}
We turn now to the full time-dependent problem \eqref{eq:fracdiff}.  For spatial discretization we
use the approximate fractional Laplacian just as in \cref{sec:elliptic-disc}, which we combine with
a Crank-Nicolson scheme for the discretization of temporal derivatives.  This leads to the implicit
time-stepping method
\begin{align}\label{eq:timestepping}
\left(\m{I} + \frac{\Delta t}{2}\m{M}\right)\v{u}^{(k+1)} &= \left(\m{I} - \frac{\Delta t}{2}\m{M}\right)\v{u}^{(k)} + \frac{\Delta t}{2}\left(\v{f}^{(k+1)} + \v{f}^{(k)} \right)
\end{align}
to be solved for $\v{u}^{(k+1)}\in\R^N$, where $\m{M}$ is as in \cref{sec:elliptic-disc} and now
$\v{u}^{(k)}\in\R^N$ and $\v{f}^{(k)}\in\R^N$ have entries $\{u^{(k)}_\v{j}\} \approx
\{u(\v{y}_\v{j}, t_k)\}$ and $\{f^{(k)}_\v{j}\} = \{f(\v{y}_\v{j}, t_k)\}$ for $t_k=k\Delta t$.

Just as in \cref{sec:elliptic-disc}, we exploit BTTB structure to apply $\m{M}$ such that
\eqref{eq:timestepping} may be solved efficiently with CG at each time step.  Compared to the
steady-state problem, the system matrix $\left(\m{I} + \frac{\Delta t}{2}\m{M}\right)$ here is much
better conditioned due to the addition of the identity.  However, we still find that the number of
iterations is reduced substantially via preconditioning, where we use the matrix $\m{I}+
\frac{\Delta t}{2}C_{\alpha,d}h^d\m{L}$ as preconditioner.

\section{Numerical results}
\label{sec:results}

To demonstrate and profile our approach to discretizing and solving fractional diffusion problems on
bounded Cartesian domains, we implemented a number of examples. All computations were performed in MATLAB R2017a on a 64-bit Ubuntu
laptop with a dual-core Intel Core i7-7500U processor at 2.70 GHz and 16GB of RAM.  All reported timings are in seconds.

  \subsection{Elliptic examples in one dimension}

\subsubsection*{A relatively smooth solution}
We begin with a one-dimensional elliptic example on the interval $\Omega=[-1,1]$, discretizing and
solving \eqref{eq:fracelliptic} with right-hand side
\begin{align}\label{eq:f_known}
  f(x) = {}_2{F}_1\left(\frac{1+\alpha}{2}, -2; \frac{1}{2};x^2\right).
\end{align}
In this case,
the analytic solution on $\Omega$ is known and is given up to a known constant of proportionality
$K_{\alpha}$ by $u(x) = K_{\alpha}^{-1}\, (1-x^2)^{2+\frac{\alpha}{2}}$ \cite{barenblatt}.  We observe that this solution is relatively smooth when extended to $\R$ due to vanishing second derivatives as $x\to\pm1$.

We discretize the interval  $\Omega$ with regularly-spaced points as in \eqref{eq:linearelliptic} with $d=1$,  choosing
$\delta$ in \eqref{eq:window} as a function of the number of discretization points $N$, such that
$w$ is supported inside a ball with a radius of 20 discretization points.  The time $t_\text{con}$ to construct the discrete operator $\m{M}$ is less than $3 \text{ms}$ in all cases for the one-dimensional case.

Using the known solution
$u(x)$ for right-hand side (RHS) \eqref{eq:f_known}, we measure the apply error of our discretization as
$e_\text{app}\equiv{\|\m{M}\v{u}_\text{true} - \v{f}\|}/{\|\v{f}\|},$ where $\v{u}_\text{true}$ is
the analytic solution sampled on the discrete grid points and $\v{f}$ is the discretized RHS.  To
demonstrate the solution error of our discretization scheme we take the same RHS as before and use
CG to solve the resulting linear system \eqref{eq:matelliptic}.  This gives a discrete solution
$\v{u}$ that we can compare to $\v{u}_\text{true}$ by computing the relative solution error
$e_\text{sol} \equiv {\|\v{u} - \v{u}_\text{true}\|}/{\|\v{u}_\text{true}\|}.$ These metrics are all
shown in \cref{tab:1D_con} for four different choices of $\alpha$, with correponding plots
in \cref{fig:1derrorplots}.  For convenience, at the bottom of \cref{tab:1D_con} we give an estimate of the asymptotic decay rate of the error as $N$ is increased, given by a least-squares fit of the log-error to log $N$.

We show in \cref{tab:1D_cg} the runtime $t_\text{CG}$ and iterations $n_\text{CG}$ required by CG to
solve the linear system \eqref{eq:matelliptic} for two different choices of  relative $\ell_2$-norm residual tolerance
$\epsilon_\text{res}$.  We give results and approximate rates of runtime growth for both the
preconditioned system (where the preconditioner $\m{L}$ is a finite-difference Laplacian as
described in \cref{sec:elliptic-disc}) and the unpreconditioned system.  Because this is a
one-dimensional problem, use of a fast Poisson solver to apply $\m{L}^{-1}$ is not stricly necessary
for efficiency.  Instead, we use a sparse Cholesky factorization, with negligible overhead. The
corresponding timing results are plotted in \cref{fig:2d3d} (left), where we see that our simple
preconditioning scheme is effective for reducing the time to solution, especially for larger
$\alpha$.

\subsubsection*{A less smooth solution}
As a second one-dimensional example, we follow Huang and Oberman \cite[Section 7]{huang-oberman-new} and take  a RHS corresponding to $f(x)=1$.  This leads to an analytic solution on $\Omega$ given by (up to known constant $K'_\alpha$)
\begin{align}\label{eq:lesssmooth}
u(x) = K'_\alpha (1-x^2)^{\alpha/2},
\end{align}
which when extended to $\R$ is only continuous as $x\to\pm1$, in contrast to the previous example.  

As in Huang and Oberman, applying the discrete forward operator $\m{M}$ to \eqref{eq:lesssmooth} is inaccurate near the boundary due to the lack of differentiability (not shown).
 However, taking $\v{f}=1$ as the RHS in the discretization of \eqref{eq:fracelliptic}, we still observe steady convergence of the relative solution error $e_\text{sol}$ as $N$ increases in \cref{tab:1D_erronly2}, though due to reduced regularity of the solution the observed rate of convergence deteriorates to $O(N^\gamma)$ with $\gamma\approx \min(1,1/2+\alpha/2)$.

\begin{table}
  \small
  \ra{1.0}
  \centering
  \caption{Relative apply and solve
    errors for $\alpha\in\{0.75,1.25,1.50,1.75\}$ for the one-dimensional elliptic example with right-hand side \eqref{eq:f_known}.  The last row gives an estimate of the rate of growth as $N$ is increased, i.e., $\gamma$ in $\O(N^\gamma)$.\label{tab:1D_con}}
  \begin{tabular}{@{}@{\extracolsep{-1.25pt}}ccccccccc@{}}\toprule
    \multicolumn{1}{c}{ $N$} &\multicolumn{1}{c}{$e_\text{app,0.75}$}
    &\multicolumn{1}{c}{$e_\text{app,1.25}$} &\multicolumn{1}{c}{$e_\text{app,1.5}$} &
    \multicolumn{1}{c}{$e_\text{app,1.75}$} & \multicolumn{1}{c}{$e_\text{sol,0.75}$}&\multicolumn{1}{c}{$e_\text{sol,1.25}$}
    &\multicolumn{1}{c}{$e_\text{sol,1.5}$} & \multicolumn{1}{c}{$e_\text{sol,1.75}$}\\ \midrule
    
    $511$   &$\scinote{2.1}{-}{07}$& $\scinote{3.6}{-}{06}$  & $\scinote{9.5}{-}{06}$  & $\scinote{2.0}{-}{05}$ &$\scinote{1.2}{-}{08}$& $\scinote{1.6}{-}{07}$& $\scinote{5.5}{-}{07}$&$\scinote{2.0}{-}{06}$ \\
    
    $1023$  &$\scinote{4.7}{-}{08}$&  $\scinote{9.6}{-}{07}$& $\scinote{2.7}{-}{06}$  & $\scinote{5.5}{-}{06}$ &$\scinote{1.7}{-}{09}$&$\scinote{2.6}{-}{08}$ & $\scinote{1.0}{-}{07}$&$\scinote{4.3}{-}{07}$  \\
    
    $2047$  &$\scinote{1.1}{-}{08}$& $\scinote{2.6}{-}{07}$ & $\scinote{7.7}{-}{07}$  & $\scinote{1.6}{-}{06}$&$\scinote{2.4}{-}{10}$&$\scinote{4.0}{-}{09}$ &$\scinote{1.8}{-}{08}$ & $\scinote{9.0}{-}{08}$ \\
    
    $4095$  &$\scinote{2.4}{-}{09}$&$\scinote{7.0}{-}{08}$ & $\scinote{2.2}{-}{07}$  & $\scinote{5.0}{-}{07}$ &$\scinote{3.3}{-}{11}$&$\scinote{6.3}{-}{10}$ & $\scinote{3.2}{-}{09}$&$\scinote{1.9}{-}{08}$
    \\ \midrule
    Rate:& -2.1& -1.9&-1.8&-1.8&-2.8&-2.7&-2.5&-2.2\\
    \bottomrule
  \end{tabular}
\end{table}

\begin{figure}
\centering
\includegraphics[scale=0.35]{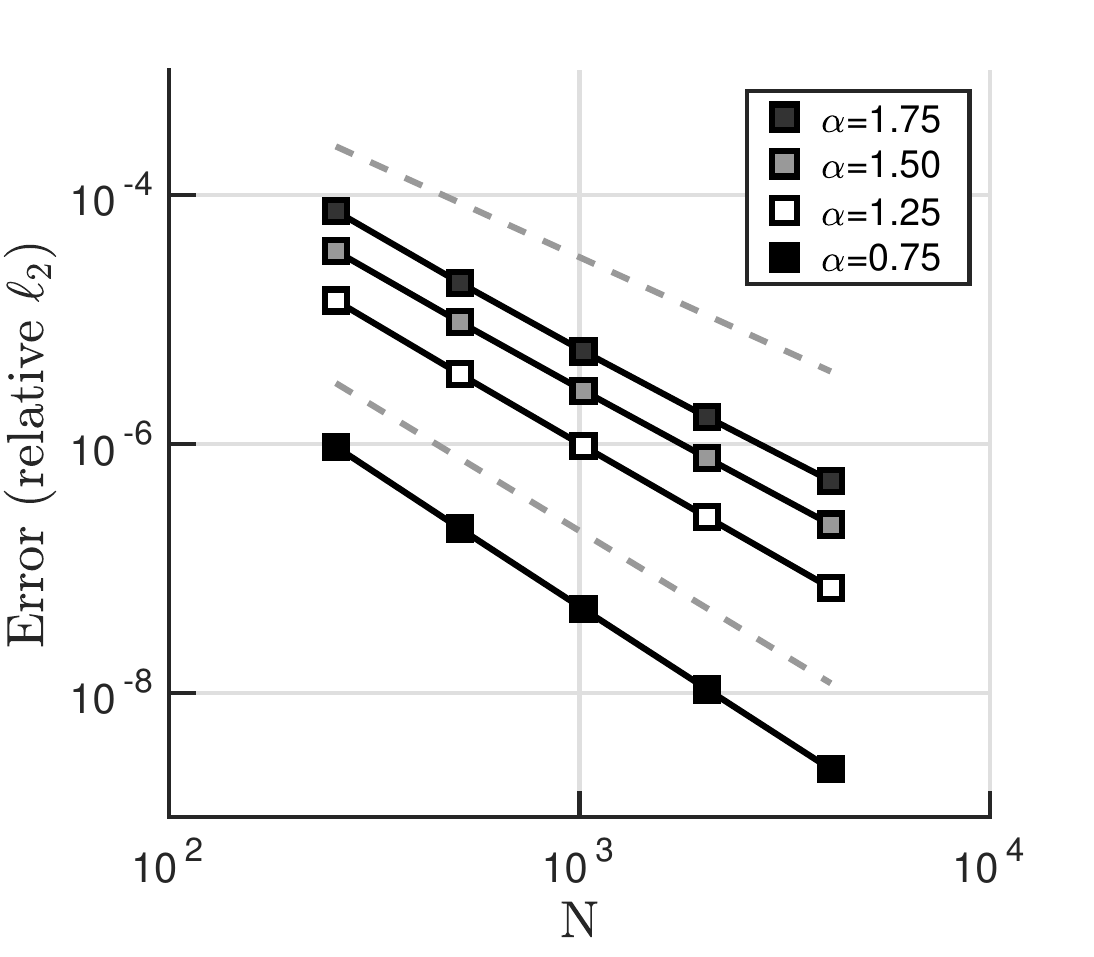}\includegraphics[scale=0.35]{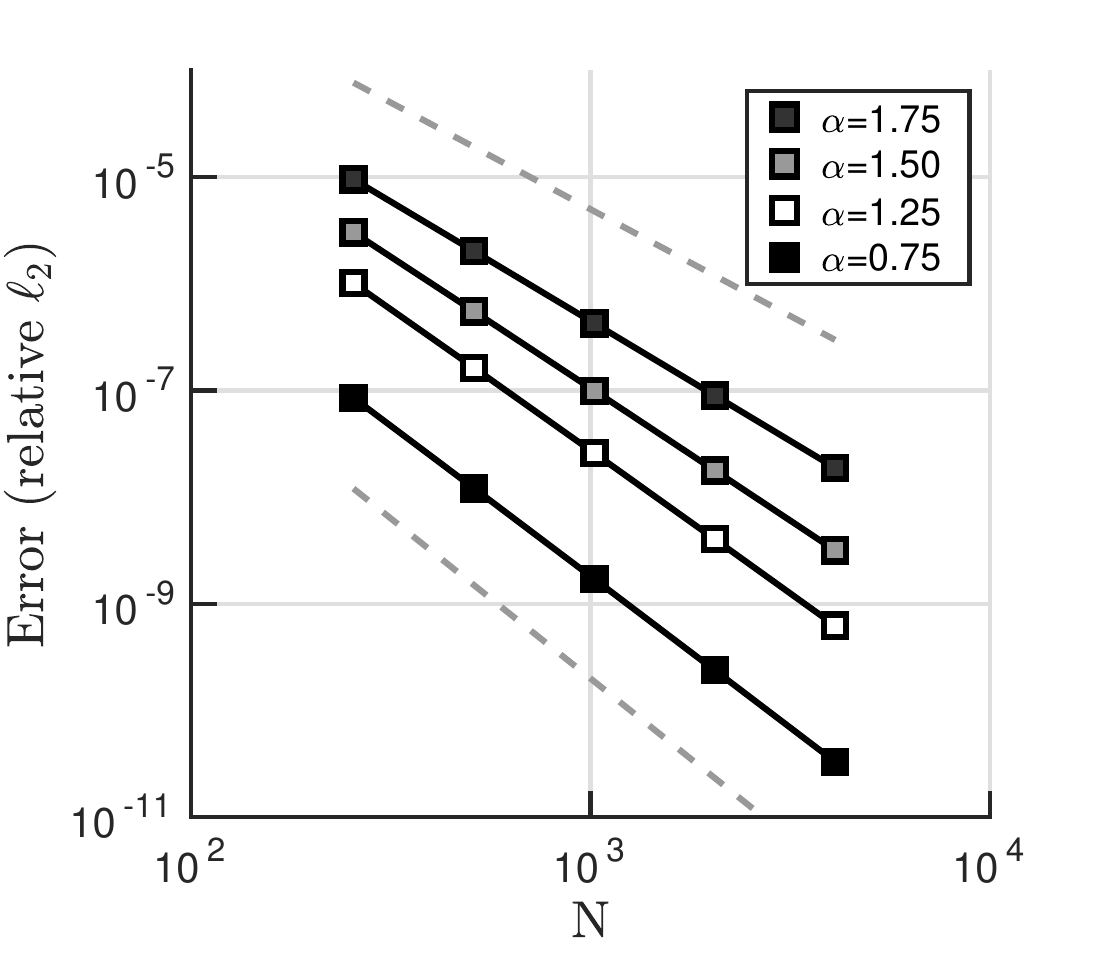}
\caption{\label{fig:1derrorplots} For the one-dimensional example, we plot the relative $\ell_2$
  apply error $e_\text{app}$ (left) and solve error $e_\text{sol}$ (right) as tabulated in
  \cref{tab:1D_con}.  In each case we see steady convergence, though with differing rates (note the
  difference in $y$-axis scale between figures).  On the left, the top trend line is $O(N^{-1.5})$ and the bottom is $O(N^{-2})$.  On the right, the top trend line is $O(N^{-2})$ and the bottom is $O(N^{-3})$.}
\end{figure}

\begin{table}
  \small
  \ra{1.0}
  \centering
  
  \caption{Runtime $t_\text{CG}$ and number of iterations $n_\text{CG}$ required to solve the
    one-dimensional elliptic example using CG with/without preconditioning based on the
    finite-difference Laplacian.  The parenthesized quantities indicate the corresponding test did
    not converge within 1000 iterations. We omit results for $\alpha=0.75$, as for $\alpha<1$ our preconditioning scheme is unnecessary.\label{tab:1D_cg}}
  \begin{tabular}{@{}cccccc@{}}\toprule
    \multicolumn{2}{c}{}&\multicolumn{2}{c}{$\epsilon_\text{res}=10^{-6}$} & \multicolumn{2}{c}{$\epsilon_\text{res}=10^{-9}$}\\
     \cmidrule(lr){3-4} \cmidrule(lr){5-6} 
     $\alpha$ & $N$ & \multicolumn{1}{c}{$t_\text{CG}$} & \multicolumn{1}{c}{$n_\text{CG}$}  &  \multicolumn{1}{c}{$t_\text{CG}$} & \multicolumn{1}{c}{$n_\text{CG}$}  \\
     \midrule
    \multirow{5}{*}{$1.25$}
    &$511$   & $\scinote{4.8}{-}{3} \,/\, \scinote{1.8}{-}{2}$ & $22 \,/\, 104$ 
    & $\scinote{7.2}{-}{3} \,/\, \scinote{2.0}{-}{2}$&  $35 \,/\, 123$     \\

    &$1023$  & $\scinote{1.4}{-}{2} \,/\, \scinote{5.7}{-}{2}$ & $27 \,/\, 162$ 
    & $\scinote{1.5}{-}{2} \,/\, \scinote{5.0}{-}{2}$ &  $44 \,/\, 190$  \\
    
    &$2047$  &$\scinote{2.6}{-}{2} \,/\, \scinote{1.5}{-}{1}$ & $35 \,/\, 251$
    & $\scinote{3.9}{-}{2} \,/\, \scinote{1.8}{-}{1}$ &  $57 \,/\, 295$ \\
    
    &$4095$  & $\scinote{6.3}{-}{2} \,/\, \scinote{4.3}{-}{1}$ & $43 \,/\, 389$
    & $\scinote{9.5}{-}{2} \,/\, \scinote{5.1}{-}{1}$ &  $72 \,/\, 457$ 
    \\\cmidrule(lr){2-6}
    &Rate:& $1.2\,/\,1.5$& * &$1.3\,/\,1.6$&*
    \\
  \midrule[\heavyrulewidth]

      \multirow{5}{*}{$1.50$}
    &$511$  &  $\scinote{3.6}{-}{3} \,/\, \scinote{2.5}{-}{2}$ & $15 \,/\, 156$
    & $\scinote{6.4}{-}{3} \,/\, \scinote{2.9}{-}{2}$ &  $23 \,/\, 174$ \\
    
    &$1023$  & $\scinote{8.1}{-}{3} \,/\, \scinote{6.9}{-}{2}$ & $20 \,/\, 263$
    & $\scinote{1.2}{-}{2} \,/\, \scinote{8.8}{-}{2}$ &  $28 \,/\, 294$ 
    \\
    &$2047$ & $\scinote{1.8}{-}{2} \,/\, \scinote{2.7}{-}{1}$ & $22 \,/\, 445$
    & $\scinote{2.9}{-}{2} \,/\, \scinote{3.1}{-}{1}$ &  $34 \,/\, 497$
    
    \\
    &$4095$ & $\scinote{3.7}{-}{2} \,/\, \scinote{8.5}{-}{1}$ & $26 \,/\, 752$ 
    & $\scinote{5.7}{-}{2} \,/\, \scinote{9.7}{-}{1}$ &  $40 \,/\, 839$
    \\ \cmidrule(lr){2-6} 
     &Rate:& $1.1\,/\,1.7$& * &$1.1\,/\,1.7$&*
    \\
  \midrule[\heavyrulewidth]
      \multirow{5}{*}{$1.75$}
    &$511$  &  $\scinote{2.7}{-}{3} \,/\, \scinote{3.3}{-}{2}$ & $11 \,/\, 216$
    & $\scinote{3.7}{-}{3} \,/\, \scinote{3.8}{-}{2}$ &  $15 \,/\, 229$\\
    
    &$1023$  & $\scinote{4.5}{-}{3} \,/\, \scinote{9.9}{-}{2}$ & $12 \,/\,397$
    & $\scinote{7.9}{-}{3} \,/\, \scinote{1.1}{-}{1}$ &  $17\,/\, 422$
    \\
    &$2047$ & $\scinote{1.1}{-}{2} \,/\, \scinote{4.3}{-}{1}$ & $13 \,/\, 731$
    & $\scinote{1.7}{-}{2} \,/\, \scinote{4.8}{-}{1}$ &  $18 \,/\, 776$
    
    \\
    &$4095$ & $\scinote{2.7}{-}{2} \,/\, (\scinote{1}{+}{0})$& $15 \,/\, (1000)$
    & $\scinote{4.2}{-}{2} \,/\, (\scinote{1}{+}{0})$ &  $21 \,/\, (1000)$ 
    \\ \cmidrule(lr){2-6}
     &Rate:& $1.1\,/\,1.8$ & * &$1.2\,/\,1.8$&*
    \\
    \bottomrule\\
  \end{tabular}
\end{table}

\begin{table}
\ra{1.0}
  \centering
  \caption{Relative solve
    errors for $\alpha\in\{0.25,0.50,0.75,1.00,1.25,1.50,1.75\}$ for the one-dimensional elliptic example with right-hand side $\v{f}=1$ and discrete solution $\v{u}$ approximating \eqref{eq:lesssmooth}.\label{tab:1D_erronly2}}
  \begin{tabular}{@{}@{\extracolsep{-1.25pt}}cccccccc@{}}\toprule
    \multicolumn{1}{c}{ $N$} &\multicolumn{1}{c}{$e_\text{sol,0.25}$}&\multicolumn{1}{c}{$e_\text{sol,0.50}$}& \multicolumn{1}{c}{$e_\text{sol,0.75}$}&\multicolumn{1}{c}{$e_\text{sol,1.00}$}&\multicolumn{1}{c}{$e_\text{sol,1.25}$}
    &\multicolumn{1}{c}{$e_\text{sol,1.50}$} & \multicolumn{1}{c}{$e_\text{sol,1.75}$}\\ \midrule
    
    $511$ &$\scinote{3.2}{-}{03}$ & $\scinote{3.5}{-}{03}$&$\scinote{3.0}{-}{03}$&$\scinote{2.5}{-}{03}$& $\scinote{2.0}{-}{03}$  & $\scinote{1.4}{-}{03}$  & $\scinote{8.1}{-}{04}$  \\
    
    $1023$ &$\scinote{2.1}{-}{03}$ &$\scinote{2.1}{-}{03}$&$\scinote{1.7}{-}{03}$&$\scinote{1.3}{-}{03}$&  $\scinote{1.0}{-}{03}$& $\scinote{7.2}{-}{04}$  & $\scinote{4.1}{-}{04}$   \\
    
    $2047$  &$\scinote{1.4}{-}{03}$&$\scinote{1.2}{-}{03}$&$\scinote{9.2}{-}{04}$&$\scinote{6.8}{-}{04}$& $\scinote{5.0}{-}{04}$ & $\scinote{3.6}{-}{04}$  & $\scinote{2.0}{-}{04}$ \\
    
    $4095$  &$\scinote{8.8}{-}{04}$&$\scinote{7.4}{-}{04}$&$\scinote{5.0}{-}{04}$&$\scinote{3.5}{-}{04}$&$\scinote{2.6}{-}{04}$ & $\scinote{1.8}{-}{04}$  & $\scinote{1.0}{-}{04}$ 
    \\ \midrule
    Rate:&-0.63&-0.75&-0.86&-0.93&-0.98&-0.99&-1.00\\
    \bottomrule
  \end{tabular}
  \end{table}

\begin{figure}
\centering \includegraphics[scale=0.35]{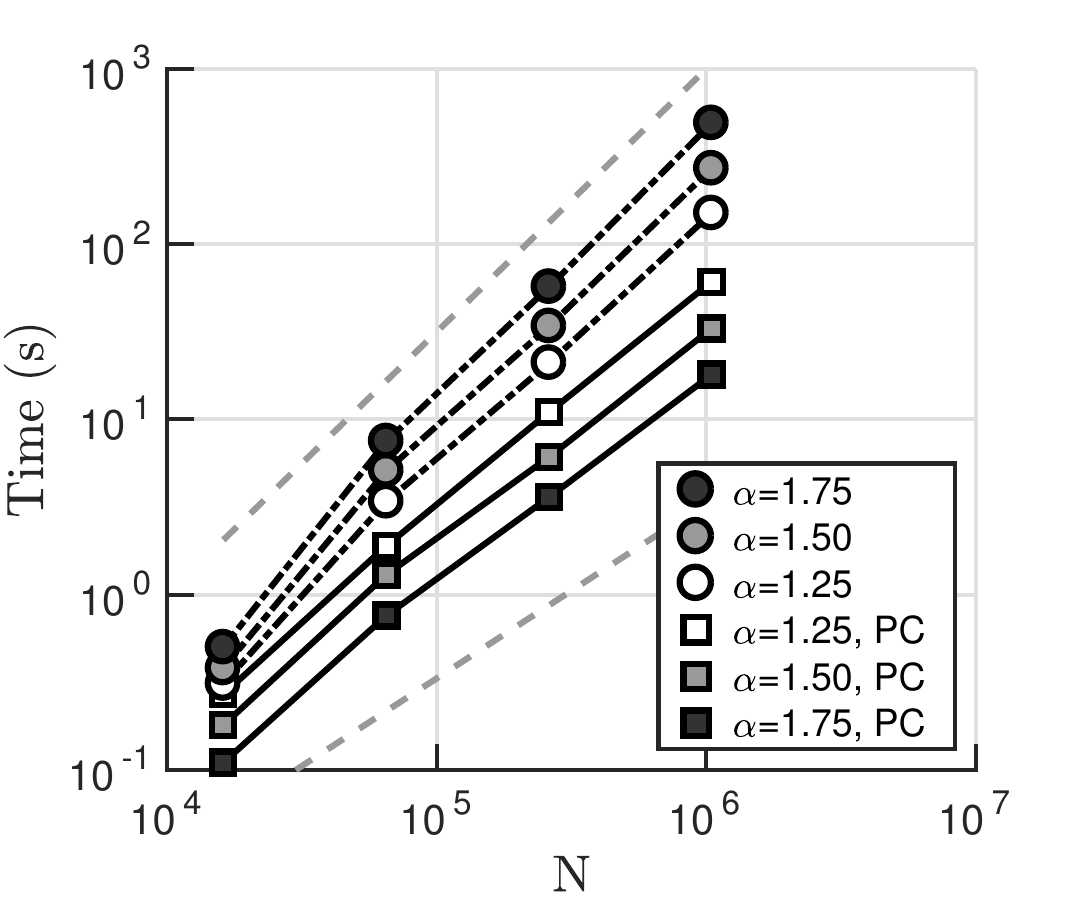}\includegraphics[scale=0.35]{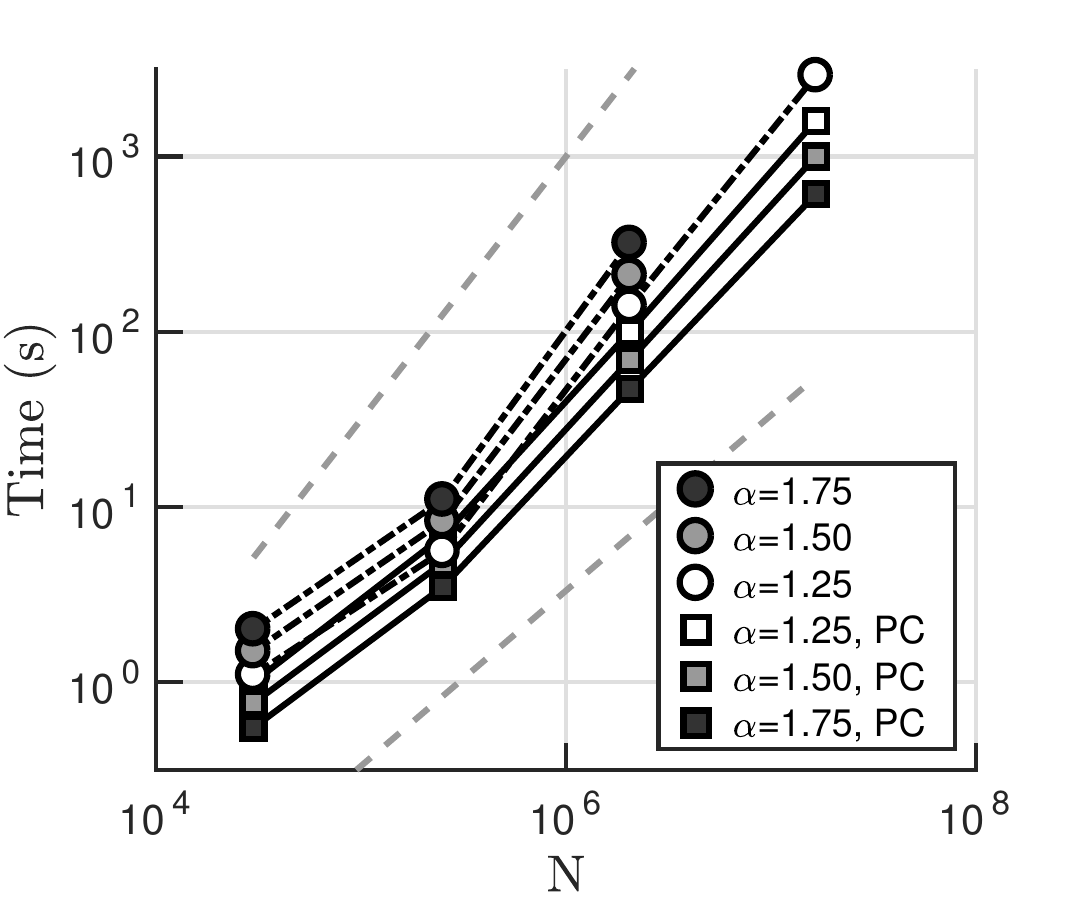}
\caption{\label{fig:2d3d} For the two-dimensional (left) and three-dimensional (right) examples, we
  plot the runtime $t_\text{CG}$ required for CG to attain an accuracy of $\epsilon_\text{res}=10^{-9}$ as tabulated in \cref{tab:2dcg,tab:3D_its}, both
  with (square markers) and without (circular markers) preconditioning.  Note that some points in the right plot are absent due to
  excessive runtime.  In both plots the top trend line is $O(N^{1.5})$ and the bottom is $O(N)$.}
\end{figure}

\subsection{Elliptic example in two dimensions: square domain}\label{sec:elliptic2d}

For a two-dimensional example, we use a square domain $\Omega~=~[0,1]^2$ discretized with a regular grid of $N$ DOFs, showing the time to construct ($t_\text{con}$) and apply ($t_\text{app}$) the discrete operator $\m{M}$  in \cref{tab:2dcon}.  For these and the remainder of our examples, we focus on the case $\alpha>1$ for brevity, as for $\alpha<1$ the linear system \eqref{eq:linearelliptic} may be solved efficiently without any preconditioning.

Unlike the
one-dimensional case, in two dimensions there is no RHS $f$ for which \eqref{eq:fracelliptic} has a
simple known solution.  Instead, we use
standard grid error estimates based on Richardson extrapolation to estimate the rate of convergence. Concretely, for our application error grid estimate we use the function
\begin{align}\label{eq:test}
g_1(\v{x}) &= \prod_{i=1}^2 \frac{1}{4}(1+\cos(2\pi x_i - \pi))^2,
\end{align}
which is nice when truncated to $\Omega=[0,1]^2$.  Using a coarse grid of size $255^2$, a
medium-scale grid of size $511^2$, and a fine-scale grid of size $1023^2$, we obtained three
corresponding estimates of the fractional Laplacian of $g_1$ evaluated on the common coarse grid:
$\{\v{f}_\text{c}, \v{f}_\text{m}, \v{f}_\text{f}\}\subset\R^{255^2}$.  The Richardson error rate
estimate is then given by
\begin{align}\label{eq:rich}
R_p  \equiv \frac{\log\|\v{f}_\text{f} - \v{f}_\text{m} \|_p - \log \|\v{f}_\text{m} - \v{f}_\text{c} \|_p}{\log 1/2},
\end{align}
where $\|\cdot\|_p$ is the $\ell_p$ norm.  For solution error, we obtain analogous error rate
estimates for the solution $u$ to the extended Dirichlet problem using RHS $g_2(\v{x})=1$.  These
rate estimates can be seen in \cref{tab:2drich}, where we observe that the solution error rates are
empirically limited to first-order due again to the general lack of smoothness of $u$ near the boundary $\partial\Omega$
\cite{bdreg}.

In \cref{tab:2dcg} and \cref{fig:2d3d} (left) we give CG convergence results for the square
example, analogous to the one-dimensional results in \cref{tab:1D_cg}.  Note that unlike the
one-dimensional case, here it is clearly advantageous to use a fast Poisson solver to apply the preconditioner.
While the reduced number of iterations is roughly offset by the cost of applying the preconditioner
at each iteration for smaller $\alpha$ and $N$, the utility of our preconditioning approach becomes
clear for the larger, more ill-conditioned problems.

\begin{table}
  \ra{1.0}
  \centering
  
  \begin{minipage}[c]{.4\linewidth}
    \centering
    \small
    \caption{Runtimes $t_\text{con}$ for the construction of the operator $\m{M}$ and $t_\text{app}$
      for application via FFT for the two-dimensional elliptic example.\label{tab:2dcon}}
    \begin{tabular}{@{}@{\extracolsep{-1pt}}ccc@{}}\toprule
      \multicolumn{1}{c}{ $N$} & \multicolumn{1}{c}{$t_\text{con}$}
      &\multicolumn{1}{c}{$t_\text{app}$} \\ \midrule
      
      $127^2$   & $\scinote{8.4}{-}{2}$ & $\scinote{4.5}{-}{3}$     \\
      
      $255^2$  & $\scinote{9.8}{-}{2}$ & $\scinote{3.0}{-}{2}$\\
      
      $511^2$  & $\scinote{1.9}{-}{1}$ & $\scinote{1.1}{-}{1}$  \\
      
      $1023^2$  & $\scinote{6.1}{-}{1}$ & $\scinote{5.0}{-}{1}$    \\\midrule
      Rate:&0.5 &1.1\\
    \bottomrule
  \end{tabular}
  
  \end{minipage}
  \begin{minipage}[c]{.4\linewidth}
    \small
    \centering
    \caption{Grid error estimates $R_2$ and $R_\infty$ for both $\v{u}$ and $\v{f}$ in
      \eqref{eq:matelliptic} for the two-dimensional elliptic example. \label{tab:2drich}}
    \begin{tabular}{@{}@{\extracolsep{-1pt}}lcccc@{}}\toprule
      \multicolumn{1}{c}{}&\multicolumn{2}{c}{Grid rate, $\v{u}$} & \multicolumn{2}{c}{Grid rate, $\v{f}$}\\
      \cmidrule(lr){2-3} \cmidrule(lr){4-5} 
      \multicolumn{1}{c}{$\alpha$} & $R_2$ &\multicolumn{1}{c}{$R_\infty$} &\multicolumn{1}{c}{$R_2$} &$R_\infty$ \\\midrule
      $1.25$   & $0.98$  & $0.87$& $2.65$    & $2.63$     \\
      
      $1.50$  & $0.99$ & $0.82$ & $2.38$    & $2.35$ \\
      
      $1.75$  & $0.99$  & $0.87$& $2.16$    & $2.15$\\
      \bottomrule
    \end{tabular}
  \end{minipage}
\end{table}

\begin{table}
  \small
  \ra{1.0}
  \centering
  
  \caption{Runtime $t_\text{CG}$ and number of iterations $n_\text{CG}$ required to solve the
    two-dimensional elliptic example using CG with/without preconditioning.  \label{tab:2dcg}}
  \begin{tabular}{@{}cccccc@{}}\toprule
    \multicolumn{2}{c}{}&\multicolumn{2}{c}{$\epsilon_\text{res}=10^{-6}$} & \multicolumn{2}{c}{$\epsilon_\text{res}=10^{-9}$}\\
    \cmidrule(lr){3-4} \cmidrule(lr){5-6} 
    $\alpha$ & $N$ & \multicolumn{1}{c}{$t_\text{CG}$} & \multicolumn{1}{c}{$n_\text{CG}$}  &  \multicolumn{1}{c}{$t_\text{CG}$} & \multicolumn{1}{c}{$n_\text{CG}$}\\  
    \midrule
    \multirow{5}{*}{$1.25$}
    &$127^2$   & $\scinote{1.7}{-}{1} \,/\, \scinote{2.2}{-}{1}$ & $21 \,/\, 61$ 
    & $\scinote{2.7}{-}{1} \,/\, \scinote{3.1}{-}{1}$&  $33 \,/\, 77$ 
    \\
    &$255^2$  & $\scinote{1.1}{+}{0} \,/\, \scinote{2.3}{+}{0}$ & $28 \,/\, 95$ 
    & $\scinote{1.9}{+}{0} \,/\, \scinote{3.4}{+}{0}$ &  $43 \,/\, 121$  \\
    
    &$511^2$  &$\scinote{6.5}{+}{0} \,/\, \scinote{1.6}{+}{1}$ & $36 \,/\, 148$
    & $\scinote{1.1}{+}{1} \,/\, \scinote{2.1}{+}{1}$ &  $57 \,/\, 188$ \\
    
    &$1023^2$ &$\scinote{3.7}{+}{1} \,/\, \scinote{1.2}{+}{2}$ & $47 \,/\, 231$
    & $\scinote{6.1}{+}{1} \,/\, \scinote{1.5}{+}{2}$ & $74 \,/\, 292$
    \\\cmidrule(lr){2-6}
&Rate:&$1.3\,/\,1.5$&*&$1.3\,/\,1.5$&*
    \\\midrule[\heavyrulewidth]
        \multirow{5}{*}{$1.50$}
    &$127^2$  &  $\scinote{1.3}{-}{1} \,/\, \scinote{3.3}{-}{1}$ & $15 \,/\, 86$
    & $\scinote{1.8}{-}{1} \,/\, \scinote{3.8}{-}{1}$ &  $23 \,/\, 108$ \\
    
    &$255^2$  & $\scinote{8.6}{-}{1} \,/\, \scinote{4.0}{+}{0}$ & $18 \,/\, 147$
    & $\scinote{1.3}{+}{0} \,/\, \scinote{5.1}{+}{0}$ &  $28 \,/\, 184$ 
    \\
    &$511^2$ & $\scinote{4.0}{+}{0} \,/\, \scinote{2.7}{+}{1}$ & $21 \,/\, 250$
    & $\scinote{6.1}{+}{0} \,/\, \scinote{3.4}{+}{1}$ &  $33 \,/\, 312$
    \\
    &$1023^2$ & $\scinote{2.2}{+}{1} \,/\, \scinote{2.2}{+}{2}$ & $26 \,/\, 425$
    &  $\scinote{3.3}{+}{1} \,/\, \scinote{2.7}{+}{2}$ & $40 \,/\, 528$
     \\\cmidrule(lr){2-6}
&Rate:&$1.2\,/\,1.5$&*&$1.2\,/\,1.6$&*
    \\\midrule[\heavyrulewidth]    
    \multirow{5}{*}{$1.75$}
    &$127^2$  &  $\scinote{8.6}{-}{2} \,/\, \scinote{4.1}{-}{1}$ & $11 \,/\, 120$
    & $\scinote{1.1}{-}{1} \,/\, \scinote{5.0}{-}{1}$ &  $15 \,/\, 149$\\
    
    &$255^2$  & $\scinote{5.7}{-}{1} \,/\, \scinote{6.1}{+}{0}$ & $12 \,/\,223$
    & $\scinote{7.6}{-}{1} \,/\, \scinote{7.5}{+}{0}$ &  $17\,/\, 278$
    \\
    &$511^2$ & $\scinote{2.5}{+}{0} \,/\, \scinote{4.6}{+}{1}$ & $13 \,/\, 413$
    & $\scinote{3.6}{+}{0} \,/\, \scinote{5.7}{+}{1}$ &  $19 \,/\, 516$
    \\
    &$1023^2$ &$\scinote{1.2}{+}{1} \,/\, \scinote{4.0}{+}{2}$ & $14 \,/\, 766$ 
    &$\scinote{1.8}{+}{1} \,/\, \scinote{4.9}{+}{2}$ & $21 \,/\, 962$ 
     \\\cmidrule(lr){2-6}
&Rate:&$1.2\,/\,1.6$&*&$1.2\,/\,1.6$&*
    \\
    \bottomrule\\
  \end{tabular}
\end{table}

As discussed, the Richardson convergence results for non-smooth $\v{u}$ in \cref{tab:2drich} are limited to roughly first-order accuracy.  To validate the accuracy of our approach on smooth solutions, we generate a smooth synthetic example in 2D as follows. Taking $u \equiv g_1$ in \eqref{eq:test}, we sample $u$ on a regular grid of $4095^2$ DOFs to obtain $\v{u}^{(4095)}_{\text{true}}$.  With $\m{M}^{(4095)}$ as our discrete operator of the corresponding size, we compute $\v{f}^{(4095)} = \m{M}^{(4095)}\v{u}^{(4095)}_\text{true}$.  For a given problem size $N=n^2$ we obtain the RHS vector $\v{f}^{(n)}$ and ``true solution'' $\v{u}^{(n)}_{\text{true}}$ by appropriately subsampling $\v{f}^{(4095)}$ and $\v{u}^{(4095)}_{\text{true}}$, respectively, which then permits computing the relative $\ell_2$ error norm $e = \|\v{u}^{(n)}_\text{true} - \v{u}^{(n)}\|/\|\v{u}^{(n)}_\text{true}\|$,
where the discrete solution satisfies the problem on the smaller grid $\m{M}^{(n)}\v{u}^{(n)} = \v{f}^{(n)}$.  Results can be seen in \cref{tab:2D_smooth}, where we note that rates given are in terms of $n=\sqrt{N}$ as appropriate for error rates on a regular 2D grid.  For this problem with smooth solution, we see better error rates compared to \cref{tab:2drich}, aligned with our theory.
\begin{table}
  \small
  \ra{1.0}
  \centering
  \caption{Relative $\ell_2$ solve
    error for $\alpha\in\{0.75,1.25,1.50,1.75\}$ for the two-dimensional elliptic example with right-hand side generated numerically by sampling as described in text.  The last row gives an estimate of the rate of growth as $n=\sqrt{N} \sim h^{-1}$ is increased, i.e., $\gamma$ in $\O(n^\gamma)$.\label{tab:2D_smooth}}
  \begin{tabular}{@{}@{\extracolsep{-1.25pt}}ccccc@{}}\toprule
    \multicolumn{1}{c}{ $N$} & \multicolumn{1}{c}{$e_\text{0.75}$}&\multicolumn{1}{c}{$e_\text{1.25}$}
    &\multicolumn{1}{c}{$e_\text{1.5}$} & \multicolumn{1}{c}{$e_\text{1.75}$}\\ \midrule
        $127^2$   &$\scinote{2.1}{-}{06}$&$\scinote{9.4}{-}{06}$ & $\scinote{2.7}{-}{05}$&$\scinote{8.2}{-}{05}$ \\
    $255^2$   &$\scinote{1.6}{-}{07}$&$\scinote{1.8}{-}{06}$ & $\scinote{6.0}{-}{06}$&$\scinote{2.1}{-}{05}$ \\

    $511^2$   &$\scinote{1.5}{-}{08}$& $\scinote{2.9}{-}{07}$& $\scinote{1.2}{-}{06}$&$\scinote{4.8}{-}{06}$ \\
    
    $1023^2$  &$\scinote{1.6}{-}{09}$&$\scinote{4.3}{-}{08}$ & $\scinote{2.1}{-}{07}$&$\scinote{9.9}{-}{07}$ 
     
    \\ \midrule
    Rate:& -3.4&-2.6&-2.3&-2.1\\
    \bottomrule
  \end{tabular}
\end{table}

\subsection{Elliptic example in three dimensions}
In three dimensions for the hypercube case $\Omega=[0,1]^3$ we repeat experiments analogous to those in \cref{sec:elliptic2d}.

To compute our three-dimensional grid error estimates \eqref{eq:rich}, we use the 3D analogue of
\eqref{eq:test},
\begin{align*}
g_1(\v{x}) &= \prod_{i=1}^3 \frac{1}{4}(1+\cos(2\pi x_i - \pi))^2
\end{align*}
for the apply error and again $g_2(\v{x})\equiv 1$ for the solution error.  We use coarse,
medium-scale, and fine grids with sizes $63^3, 127^3,$ and $255^3$, respectively, and give the
results in \cref{tab:3Drich}.  We remark that the error rates reported for $\v{f}$ appear
artifically inflated, likely due to the fact that the grid error estimate is an asymptotic
approximation that holds in the limit of large $N$, and $N=63^3$ is not large.

In \cref{tab:3D_its} and \cref{fig:2d3d} (right) we give CG convergence results for the
three-dimensional example, just like those for the two-dimensional example.  Just as in two
dimensions, the utility of our simple preconditioner is clear for larger problems and for larger
$\alpha$, where iterative approaches to solving the linear system start to become prohibitively
expensive without preconditioning.

\begin{table}
  \ra{1.0}
  \centering
  
  \begin{minipage}[c]{.4\linewidth}
    \centering
    \small
    \caption{Runtimes $t_\text{con}$ for the construction of the operator $\m{M}$ and $t_\text{app}$
      for application via FFT for the three-dimensional elliptic example.\label{tab:3dcon}}
    \begin{tabular}{@{}@{\extracolsep{-1pt}}ccc@{}}\toprule
      \multicolumn{1}{c}{
        $N$} & \multicolumn{1}{c}{$t_\text{con}$} &\multicolumn{1}{c}{$t_\text{app}$} \\ \midrule
      
      $31^3$   & $\scinote{2.3}{-}{1}$ & $\scinote{3.0}{-}{2}$     \\
      
      $63^3$  & $\scinote{3.1}{-}{1}$ & $\scinote{8.7}{-}{2}$\\
      
      $127^3$  & $\scinote{1.3}{+}{0}$ & $\scinote{1.3}{+}{0}$  \\
      
      $255^3$  & $\scinote{1.8}{+}{1}$ & $\scinote{2.0}{+}{1}$\\\midrule
      Rate: & 0.7 & 1.1
      \\
      \bottomrule
    \end{tabular}
    
  \end{minipage}
  \begin{minipage}[c]{.4\linewidth}
    \small
    \centering
    \caption{Grid error estimates $R_2$ and $R_\infty$ for both $\v{u}$ and $\v{f}$ in
      \eqref{eq:matelliptic} for the three-dimensional elliptic example. \label{tab:3Drich}}
    \begin{tabular}{@{}@{\extracolsep{-1pt}}lcccc@{}}\toprule
      \multicolumn{1}{c}{}&\multicolumn{2}{c}{Grid rate, $\v{u}$} & \multicolumn{2}{c}{Grid rate, $\v{f}$}\\
      \cmidrule(lr){2-3} \cmidrule(lr){4-5} 
      \multicolumn{1}{c}{$\alpha$} & $R_2$ &\multicolumn{1}{c}{$R_\infty$} &\multicolumn{1}{c}{$R_2$} &$R_\infty$ \\\midrule
      $1.25$   & $1.02$  & $0.90$& $4.35$
      & $4.55$       \\
      $1.50$  & $1.10$ & $0.87$ & $4.22$
      & $4.62$ \\
      $1.75$  & $1.30$  & $1.13$& $3.90$
      & $4.99$\\
      \bottomrule
    \end{tabular}
  \end{minipage}
\end{table}

\begin{table}
  \small
  \ra{1.0}
  \centering
  \caption{Runtime $t_\text{CG}$ and number of iterations $n_\text{CG}$ required to solve the
    three-dimensional elliptic example using CG with/without preconditioning.  The parenthesized quantities indicate the corresponding test did
    not converge within 250 iterations. For the rate computations we omit $N=31^3$ due to the clear non-asymptotic behavior in \cref{fig:2d3d}.\label{tab:3D_its}}
  \begin{tabular}{@{}cccccc@{}}\toprule
    \multicolumn{2}{c}{}&\multicolumn{2}{c}{$\epsilon_\text{res}=10^{-6}$} & \multicolumn{2}{c}{$\epsilon_\text{res}=10^{-9}$}\\
    \cmidrule(lr){3-4} \cmidrule(lr){5-6} 
    $\alpha$ & $N$ & \multicolumn{1}{c}{$t_\text{CG}$} & \multicolumn{1}{c}{$n_\text{CG}$}  &  \multicolumn{1}{c}{$t_\text{CG}$} & \multicolumn{1}{c}{$n_\text{CG}$}\\  
    \midrule
    \multirow{5}{*}{$1.25$}
    &$31^3$   & $\scinote{6.4}{-}{1} \,/\, \scinote{7.7}{-}{1}$ & $ 13 \,/\, 27$ 
    & $\scinote{1.0}{+}{0} \,/\, \scinote{1.1}{+}{0}$&  $20 \,/\, 37$ 
    \\
    &$63^3$  & $\scinote{4.3}{+}{0} \,/\, \scinote{4.6}{+}{0}$ & $16 \,/\, 47$ 
    & $\scinote{6.8}{+}{0} \,/\, \scinote{5.6}{+}{0}$ &  $24 \,/\, 63$  \\
    
    &$127^3$  &$\scinote{6.9}{+}{1} \,/\, \scinote{1.0}{+}{2}$ & $21 \,/\, 75$
    & $\scinote{1.0}{+}{2} \,/\, \scinote{1.4}{+}{2}$ &  $32 \,/\, 100$ \\
    
    &{$255^3$} &$\scinote{1.0}{+}{3} \,/\, \scinote{2.2}{+}{3}$ & $27 \,/\, 117 $
    & $\scinote{1.6}{+}{3} \,/\, \scinote{2.9}{+}{3}$ & $42 \,/\, 156$
     \\\cmidrule(lr){2-6}
&Rate:&$1.3\,/\,1.5$&*&$1.3\,/\,1.5$&*
    \\\midrule[\heavyrulewidth]    
        \multirow{5}{*}{$1.50$}
    &$31^3$  &  $\scinote{5.3}{-}{1} \,/\, \scinote{1.1}{+}{0}$ & $10 \,/\, 36$
    & $\scinote{7.6}{-}{1} \,/\, \scinote{1.5}{+}{0}$ &  $14 \,/\, 50$ \\
    
    &$63^3$  & $\scinote{3.5}{+}{0} \,/\, \scinote{5.9}{+}{0}$ & $12 \,/\, 66$
    & $\scinote{4.8}{+}{0} \,/\, \scinote{8.3}{+}{0}$ &  $17 \,/\, 91$ 
    \\
    &$127^3$ & $\scinote{4.7}{+}{1} \,/\, \scinote{1.5}{+}{2}$ & $14 \,/\, 113$
    & $\scinote{6.9}{+}{1} \,/\, \scinote{2.1}{+}{2}$ &  $21 \,/\, 154$
    \\
    &{$255^3$} & $\scinote{7.2}{+}{2} \,/\, \scinote{3.6}{+}{3}$ & $18 \,/\, 195 $
    &  $\scinote{1.0}{+}{3} \,/\, (\scinote{4.6}{+}{3})$ & $26 \,/\, (250) $
    \\\cmidrule(lr){2-6}
&Rate:&$1.3\,/\,1.5$&*&$1.3\,/\,1.5$&*
    \\\midrule[\heavyrulewidth]    
    \multirow{5}{*}{$1.75$}
    &$31^3$  &  $\scinote{4.1}{-}{1} \,/\, \scinote{1.6}{+}{0}$ & $7 \,/\, 49$
    & $\scinote{5.5}{-}{1} \,/\, \scinote{2.0}{+}{0}$ &  $10 \,/\, 65$\\
    
    &$63^3$  & $\scinote{2.4}{+}{0} \,/\, \scinote{8.3}{+}{0}$ & $8 \,/\,93$
    & $\scinote{3.5}{+}{0} \,/\, \scinote{1.1}{+}{1}$ &  $12\,/\, 123$
    \\
    &$127^3$ & $\scinote{3.4}{+}{1} \,/\, \scinote{2.4}{+}{2}$ & $10 \,/\, 175$
    & $\scinote{4.7}{+}{1} \,/\, \scinote{3.2}{+}{2}$ &  $14 \,/\, 229$
    
    \\
    &{$255^3$} &$\scinote{4.6}{+}{2} \,/\, (\scinote{4.6}{+}{3})$ & $11 \,/\, (250) $ 
    &$\scinote{6.1}{+}{2} \,/\, (\scinote{4.6}{+}{3})$ & $15 \,/\ (250)$ 
     \\\cmidrule(lr){2-6}
&Rate:&$ 1.3\,/\,1.6$&*&$ 1.2\,/\,1.6$&*
    \\    
    \bottomrule\\
  \end{tabular}
\end{table}

\subsection{Elliptic examples in two dimensions: an ``L''-shaped domain and octagonal domain}

Before moving to time-dependent examples, we include a final demonstration showing our method applied to problems where the domain is an occluded Cartesian grid (i.e., a regular discretization that is not a hypercube).  The first such example is a problem on an ``L''-shaped domain obtained by taking a regular grid
of $(n-1)^2$ points as before and then removing $(n/2)^2$ contiguous points corresponding to a single corner of the domain, see \cref{fig:ell} (left and center).

\begin{figure}
\centering \includegraphics[scale=0.5]{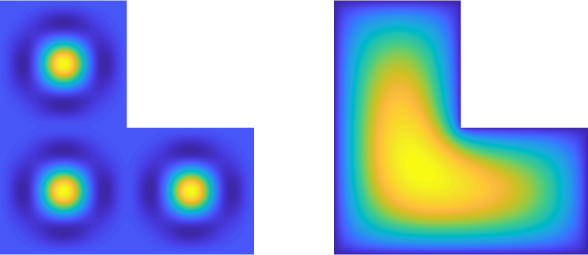} \hspace{0.8cm} \includegraphics[scale=0.095]{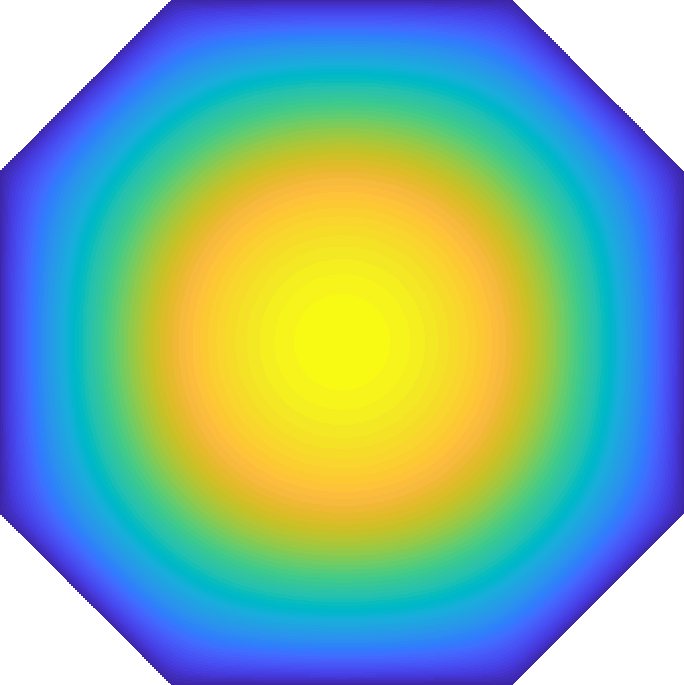}
\caption{\label{fig:ell} For a qualitative demonstration on an ``L''-shaped domain, we plot the right-hand side (left) and solution with $\alpha=1.75$ (center).  As in our other examples, the solution is forced to zero outside of the domain.  For results on an irregular octagonal domain (right), we build an octagon with horizontal and vertical sides of length 2 and all other sides of length $\sqrt{2}$. }
\end{figure}

Because the ``L''-shaped domain is discretized as a subset of a regular grid, the operator $\m{M}$ can still be applied quickly via FFT as before.  However, use of a fast Poisson solver to apply the preconditioner is no longer possible.  Instead, we use a sparse Cholesky factorization as in our one-dimensional examples.  We use the default MATLAB permutation options for sparse Cholesky (which corresponds to an approximate minimum degree ordering of the unknowns) though other methods are possible.

In \cref{tab:L2} we show results for the ``L''-shaped domain for choices of $N$ ranging from $N=12033$ (i.e., $127^2 - 64^2$) to $N=784385$ (i.e., $1023^2-512^2$).  We focus on the time $t_\text{PC}$ to construct the factored preconditioner using sparse Cholesky as well as the runtime and number of iterations for CG both with and without our preconditioning scheme, as before.  For brevity we give only results for $\alpha=1.75$, as results for smaller $\alpha$ follow the same trends as in the case of a square domain.  We remark that, while the time to factor the preconditioner is nonzero, it is still small compared to the time to solve the systems with CG, and the runtimes $t_\text{CG}$ in \cref{tab:L2} are comparable to those in \cref{tab:2dcg} (albeit one must adjust for the slighly different system sizes).

\begin{table}
  \small
  \ra{1.0}
  \centering
  \caption{Runtime $t_\text{PC}$ for construction of the preconditioner for the two-dimensional elliptic example on the ``L''-shaped domain, as well as time $t_\text{CG}$ and number of iterations $n_{CG}$ required for CG to converge to tolerances $10^{-6}$ and $10^{-9}$ with/without preconditioning for the case $\alpha=1.75$.  The parenthesized quantities indicate the corresponding test did
    not converge within 1000 iterations.\label{tab:L2}}
  \begin{tabular}{@{}ccccccc@{}}\toprule
    \multicolumn{3}{c}{}&\multicolumn{2}{c}{$\epsilon_\text{res}=10^{-6}$} & \multicolumn{2}{c}{$\epsilon_\text{res}=10^{-9}$}\\
    \cmidrule(lr){4-5} \cmidrule(lr){6-7} 
    $\alpha$ & $N$ & $t_\text{PC}$ &\multicolumn{1}{c}{$t_\text{CG}$} & \multicolumn{1}{c}{$n_\text{CG}$}  &  \multicolumn{1}{c}{$t_\text{CG}$} & \multicolumn{1}{c}{$n_\text{CG}$}\\  
    \midrule
    \multirow{5}{*}{$1.75$}
    &$12033$  & $\scinote{1.9}{-}{2}$ &$\scinote{6.2}{-}{2} \,/\, \scinote{4.9}{-}{1}$ & $10 \,/\, 141$
    & $\scinote{8.0}{-}{2} \,/\, \scinote{6.4}{-}{1}$ &  $14 \,/\, 191$\\
    
    &$48641$  &$\scinote{8.4}{-}{2}$ &$\scinote{4.3}{-}{1} \,/\, \scinote{7.2}{+}{0}$ & $11 \,/\,268$
    & $\scinote{5.8}{-}{1} \,/\, \scinote{9.2}{+}{0}$ &  $16\,/\, 354$
    \\
    &$195585$ &$\scinote{3.9}{-}{1}$ &$\scinote{2.1}{+}{0} \,/\, \scinote{5.4}{+}{1}$ & $13 \,/\, 497$
    & $\scinote{2.7}{+}{0} \,/\, \scinote{7.0}{+}{1}$ &  $18 \,/\, 655$
    
    \\
    &{$784385$}&$\scinote{1.8}{+}{0}$ &$\scinote{1.0}{+}{1} \,/\, \scinote{4.5}{+}{2}$ & $14 \,/\, 925 $ 
    &$\scinote{1.4}{+}{1} \,/\, (\scinote{4.9}{+}{2})$ & $20 \,/\ (1000)$ 
     \\\cmidrule(lr){2-7}
&Rate:&$1.1$ &$ 1.2\,/\,1.6$&*&$ 1.2\,/\,1.7$&*
    \\    
    \bottomrule\\
  \end{tabular}
\end{table}

For another example discretized on a subset of a regular grid, we construct an irregular octagon whose vertices coincide with points on a regular grid, see \cref{fig:ell} (right).  As with the previous example, we use a sparse Cholesky factorization to build the preconditioner and give analogous results for $\alpha=1.75$ and varying $N$ and $\epsilon_\text{res}$ in \cref{tab:octagon2}.

\begin{table}
  \small
  \ra{1.0}
  \centering
  \caption{Runtime $t_\text{PC}$ for construction of the preconditioner for the two-dimensional elliptic example on the octagonal domain, as well as time $t_\text{CG}$ and number of iterations $n_{CG}$ required for CG to converge to tolerances $10^{-6}$ and $10^{-9}$ with/without preconditioning for the case $\alpha=1.75$.}\label{tab:octagon2}
  \begin{tabular}{@{}ccccccc@{}}\toprule
    \multicolumn{3}{c}{}&\multicolumn{2}{c}{$\epsilon_\text{res}=10^{-6}$} & \multicolumn{2}{c}{$\epsilon_\text{res}=10^{-9}$}\\
    \cmidrule(lr){4-5} \cmidrule(lr){6-7} 
    $\alpha$ & $N$ & $t_\text{PC}$ &\multicolumn{1}{c}{$t_\text{CG}$} & \multicolumn{1}{c}{$n_\text{CG}$}  &  \multicolumn{1}{c}{$t_\text{CG}$} & \multicolumn{1}{c}{$n_\text{CG}$}\\  
    \midrule
    \multirow{5}{*}{$1.75$}
    &$14145$  & $\scinote{5.3}{-}{2}$ &$\scinote{9.3}{-}{2} \,/\, \scinote{3.8}{-}{1}$ & $10 \,/\, 119$
    & $\scinote{1.2}{-}{1} \,/\, \scinote{4.9}{-}{1}$ &  $14 \,/\, 145$\\
    
    &$56961$  &$\scinote{1.4}{-}{1}$ &$\scinote{4.5}{-}{1} \,/\, \scinote{5.4}{+}{0}$ & $11 \,/\,220$
    & $\scinote{6.4}{-}{1} \,/\, \scinote{6.9}{+}{0}$ &  $16\,/\, 267$
    \\
    &$228609$ &$\scinote{5.5}{-}{1}$ &$\scinote{2.0}{+}{0} \,/\, \scinote{4.0}{+}{1}$ & $13 \,/\, 407$
    & $\scinote{2.8}{+}{0} \,/\, \scinote{4.9}{+}{1}$ &  $18 \,/\, 493$
    
    \\
    &{$915969$}&$\scinote{2.7}{+}{0}$ &$\scinote{1.1}{+}{1} \,/\, \scinote{3.7}{+}{2}$ & $15 \,/\, 752 $ 
    &$\scinote{1.4}{+}{1} \,/\, \scinote{4.5}{+}{2}$ & $20 \,/\ 910$ 
     \\\cmidrule(lr){2-7}
&Rate:&$0.9$ &$ 1.1 \,/\, 1.6$&*&$ 1.1 \,/\, 1.6$&*
    \\    
    \bottomrule\\
  \end{tabular}
\end{table}

\subsection{Time-dependent example in two and three dimensions}\label{sec:para}
Finally, we turn to the time-dependent case.  As described in \cref{sec:time}, our approach to the time-dependent fractional diffusion problem
\eqref{eq:fracdiff} involves first computing the discrete fractional Laplacian operator as before
and then using a Crank-Nicolson method to time-step the solution.  Here we demonstrate the
efficiency of our preconditioning scheme for the time-dependent problem and give grid error
estimates for $\Omega=[0,1]^d$.

For smooth solutions, the Crank-Nicolson scheme is locally second-order in time and our spatial
discretization is locally second-order in space.  Thus, we choose our temporal step size in $d$
dimensions as $\Delta t = (N^{1/d} + 1)^{-1}$ such that $\Delta t\approx h$ but the number of time
steps required to reach a final time of $T=0.25$ is integral.  However, we remark that, just as in
the elliptic case, we cannot expect better than first-order convergence in general \cite{real}.

For our grid error estimates we take $f\equiv 0$ in \eqref{eq:fracdiff} and initial condition
\begin{align*}
u_0(\v{x}) = \prod_{i=1}^d \frac{1}{4}(1+\cos(2\pi \nu_ix_i - \pi))^2,
\end{align*}
with $\nu_1 = 3$, $\nu_2=11$, and $\nu_3=2$.  Using \eqref{eq:timestepping} to time-step the
solution to final time $T=0.25$, we then compute grid error estimates for simultaneous refinement in
space and time, which are given in \cref{tab:richpara}.  For the two-dimensional case we use spatial
grids with $255^2$, $511^2$, and $1023^2$ points for the coarse, medium-scale, and fine grids,
respectively.  For the three-dimensional case we analogously use $31^3$, $63^3$, and $127^3$ points
in space, as we are limited by the runtime requirements of solving the largest problems.  As in the
elliptic setting, we observe an artificial inflation of the Richardson rate in three dimensions.

In \cref{tab:td_its} we give the CG results for a single time-step with random RHS, i.e., the
results for a single linear system.  In both two dimensions and three dimensions, we use the
preconditioner described in \cref{sec:discretization} applied with a modified fast Poisson solver.
Compared to the elliptic setting, we see that the time-dependent system matrix is better conditioned
and thus preconditioning for $\alpha=1.25$ is not necessary in two dimensions and not helpful in
three dimensions. However, for larger $\alpha$ there is a clear benefit.

We remark that in practical settings with multiple time steps the number of iterations is reduced
slightly from the current setting because the old solution $\v{u}^{(k)}$ can be used as an initial
guess for the solution $\v{u}^{(k+1)}$, but the difference is not substantial in general.

\begin{table}
  \small
  \centering
  \ra{1.0}
  \caption{Grid error estimates $R_2$ and $R_\infty$ for the solution at time $T=0.25$ to the parabolic problem \eqref{eq:fracdiff} in both two and three dimensions as described in \cref{sec:para}. \label{tab:richpara}}
  \begin{tabular}{@{}@{\extracolsep{-1pt}}lcccc@{}}\toprule
    \multicolumn{1}{c}{}&\multicolumn{2}{c}{2D} & \multicolumn{2}{c}{3D}\\
    \cmidrule(lr){2-3} \cmidrule(lr){4-5} 
    \multicolumn{1}{c}{$\alpha$} & $R_2$ &\multicolumn{1}{c}{$R_\infty$} &\multicolumn{1}{c}{$R_2$} &$R_\infty$ \\\midrule
    $1.25$   & $0.97$  & $0.87$& $2.22$
    & $3.11$ 
    \\
    $1.50$  & $0.97$ & $0.82$ & $2.30$
    & $3.24$ \\
    $1.75$  & $0.97$  & $0.85$& $3.02$
    & $3.39$\\
    \bottomrule
  \end{tabular}

\end{table}

\begin{table}
  \small
  \ra{1.0}
  \centering
  \caption{Runtime $t_\text{CG}$ and number of iterations $n_\text{CG}$ required to perform a single
    time step for the time-dependent example using CG with/without preconditioning in both two and three dimensions.  In all cases we use
    $\epsilon_\text{res}=10^{-9}$.  \label{tab:td_its}}
  \begin{tabular}{@{}cccccccc@{}}\toprule
    \multicolumn{1}{c}{}&\multicolumn{3}{c}{2D} & \multicolumn{3}{c}{3D}\\
    \cmidrule(lr){2-4} \cmidrule(lr){5-7} 
    $\alpha$ & $N$ & \multicolumn{1}{c}{$t_\text{CG}$} & \multicolumn{1}{c}{$n_\text{CG}$}  &  \multicolumn{1}{c}{$N$}&\multicolumn{1}{c}{$t_\text{CG}$} & \multicolumn{1}{c}{$n_\text{CG}$}\\  
    \midrule[\heavyrulewidth]
    \multirow{4}{*}{$1.25$}
    &$255^2$ & $\scinote{5.4}{-}{1} \,/\, \scinote{6.3}{-}{1}$ & $ 12\,/\, 25 $ 
    &$63^3$& $\scinote{2.9}{+}{0} \,/\, \scinote{2.1}{+}{0}$ &  $11 \,/\, 25$  \\
    
    &$511^2$  &$\scinote{2.1}{+}{0} \,/\, \scinote{2.7}{+}{0}$ & $ 12 \,/\, 27 $
    &$127^3$& $\scinote{3.5}{+}{1} \,/\, \scinote{4.5}{+}{1}$ &  $12 \,/\, 28 $ \\
    
    &{$1023^2$} &$\scinote{9.4}{+}{0} \,/\, \scinote{1.5}{+}{1}$ & $ 12 \,/\, 29  $
    &$255^3$& $\scinote{4.7}{+}{2} \,/\, \scinote{5.8}{+}{2}$ & $ 12\,/\,30 $
    \\\cmidrule(lr){2-4} \cmidrule(lr){5-7}
&Rate:&$1.0\,/\, 1.1$&*&Rate:&$1.2\,/\,1.3$&*\\
  \midrule[\heavyrulewidth]
        \multirow{4}{*}{$1.50$}
    &$255^2$  & $\scinote{5.4}{-}{1} \,/\, \scinote{1.4}{+}{0}$ & $ 12\,/\,56 $
    &$63^3$& $\scinote{2.7}{+}{0} \,/\, \scinote{3.9}{+}{0}$ &  $10 \,/\, 47 $ 
    \\
    &$511^2$ & $\scinote{2.3}{+}{0} \,/\, \scinote{6.5}{+}{0}$ & $ 13 \,/\, 67$
    &$127^3$& $\scinote{3.8}{+}{1} \,/\, \scinote{7.7}{+}{1}$ &  $ 12\,/\, 60 $
    
    \\
    &$1023^2$ & $\scinote{1.1}{+}{1} \,/\, \scinote{4.3}{+}{1}$ & $ 14 \,/\, 79  $
    &$255^3$&  $\scinote{5.0}{+}{2} \,/\, \scinote{1.4}{+}{3}$ & $ 13\,/\, 73 $
    \\\cmidrule(lr){2-4} \cmidrule(lr){5-7}
&Rate:&$1.1\,/\, 1.2$&*&Rate:&$1.2\,/\, 1.4$&*\\   \midrule[\heavyrulewidth]    \multirow{4}{*}{$1.75$}
    &$255^2$  & $\scinote{4.4}{-}{1} \,/\, \scinote{2.8}{+}{0}$ & $ 10 \,/\, 121$
    &$63^3$& $\scinote{2.2}{+}{0} \,/\, \scinote{6.5}{+}{0}$ &  $8\,/\,81 $
    \\
    &$511^2$ & $\scinote{1.9}{+}{0} \,/\, \scinote{1.5}{+}{1}$ & $11 \,/\, 162$
    &$127^3$& $\scinote{2.9}{+}{1} \,/\, \scinote{1.6}{+}{2}$ &  $ 9\,/\,122 $
    
    \\
    &{$1023^2$} &$\scinote{9.3}{+}{0} \,/\, \scinote{1.2}{+}{2}$ & $ 12 \,/\, 212 $ 
    &$255^3$&$\scinote{4.1}{+}{2} \,/\, \scinote{3.2}{+}{3}$ & $ 10\,/\, 170$ 
    \\\cmidrule(lr){2-4} \cmidrule(lr){5-7}
&Rate:&$1.1\,/\, 1.4$&*&Rate:&$1.2\,/\, 1.5$&*\\
    \bottomrule
  \end{tabular}
\end{table}

\section{Conclusions}
\label{sec:Conclusions}

We introduced a simple discretization scheme for the fractional Laplacian operator in one, two, and
three dimensions based on singularity subtraction combined with the regularly-spaced trapezoidal
rule.  When applied to sufficiently smooth functions $u$, the resulting discretization is provably
second-order accurate in the grid spacing $h$, whereas for rougher $u$ we observe first-order
accuracy in the $\ell_2$ solution error.

When the order $\alpha$ of the fractional Laplacian is close to two, the discrete operator is ill
conditioned, reflecting the underlying ill-conditioning of the continuous (integer-order) Laplacian.
To efficiently solve linear systems with the discrete fractional Laplacian, we demonstrated the
utility of a simple preconditioning scheme based on fast Poisson solvers.

For higher-order schemes, it is necessary to forsake the simplicity of our approach to more
precisely handle solutions $u$ that exhibit only fractional-order smoothness near the boundary of
$\Omega$ for both \eqref{eq:fracdiff} and \eqref{eq:fracelliptic}.  While we intend to pursue this
in future work, we have shown that the scheme presented here provides a fast, simple alternative for
situations amenable to lower-order accuracy.

\bibliographystyle{siamplain} \bibliography{fl}
\label{LastPage}
\end{document}